\documentclass[11pt, reqno]{amsart}
\usepackage{fourier,amsfonts,amssymb,graphicx,enumerate,youngtab}

\numberwithin{equation}{section}

\theoremstyle{definition}

\theoremstyle{plain}
\newtheorem{theorem}{Theorem}[section]
\newtheorem{proposition}[theorem]{Proposition}
\newtheorem{corollary}[theorem]{Corollary}
\newtheorem{lemma}[theorem]{Lemma}

\newcommand{\CC}{\mathbf{C}}
\newcommand{\HH}{\mathbf{H}}
\newcommand{\RR}{\mathbf{R}}
\newcommand{\QQ}{\mathbf{Q}}
\newcommand{\ZZ}{\mathbf{Z}}
\newcommand{\A}{\mathcal{A}}

\newcommand{\cM}{\mathcal{M}}

\newcommand{\hh}{\mathfrak{h}}

\begin{document}

\title{Cohomology of the hyperplane complement of a quaternionic reflection group}

\author{Stephen Griffeth}

\author{David Guevara}

\address{Instituto de Matem\'aticas \\
Universidad de Talca  \\}

\begin{abstract}
We compute the graded rank of the cohomology of the hyperplane complement associated with a quaternionic reflection group, and observe that it factors into irreducible factors with positive integer coefficients. For an irreducible group, these irreducible factors are all linear except at most one irreducible quadratic factor, which occurs for precisely three of the exceptional groups. 
\end{abstract}

\thanks{We thank Gwyn Bellamy, Gerhard Röhrle, Johannes Schmitt, Don Taylor, and Ulrich Thiel for comments on a preliminary version of this article.}

\maketitle

\section{Introduction} 

\subsection{Reflection groups} Each compact Lie group $K$ with maximal torus $T$ produces a \emph{Weyl group} $W$, which is a finite linear group $W=N_K(T)/T$ acting on the dual of the Lie algebra of $T$ and stabilizing the lattice of characters of $T$. This representation of the Weyl group is in particular defined over $\QQ$, and one shows that $W$ is generated by the reflections it contains. In this way compact Lie groups are related to reflection groups defined over $\QQ$, with many of the properties of $K$ encoded by the corresponding Weyl group $W$. Extending scalars from $\QQ$ to $\RR$ and considering those finite $\RR$-linear groups generated by reflections produces the class of finite Coxeter groups, and further enlarging the class of groups under study to allow a group $W$ generated by reflections acting on a $\CC$-vector space $V$ results in complex reflection groups, characterized as those finite linear groups whose corresponding quotient space $V/W$ is smooth. The sequence $\QQ \subseteq \RR \subseteq \CC$ of fields may be extended to $$\QQ \subseteq \RR \subseteq \CC \subseteq \HH$$ by considering the division ring $\HH$ of quaternions, and while the resulting class of reflection groups is far less well-studied than the class of complex reflection groups, it has recently developed (see \cite{CaGr}) that the natural generality for the study of diagonal coinvariant rings of complex reflection groups seems to be in the context of these \emph{quaternionic reflection groups} (also known as \emph{symplectic reflection groups}). The purpose of this paper is to explore another aspect of the theory of complex reflection groups that, quite mysteriously, also extends to quaternionic reflection groups. Specifically, we compute the Poincar\'e polynomials of the hyperplane complements for quaternionic reflection groups and observe that they exhibit some extremely striking behavior similar to the case of complex reflection groups (as first proved in \cite{OrSo}): they factor into irreducible (in $\ZZ[t]$) factors of degree at most two with positive integer coefficients. Our calculation relies on Cohen's classification \cite{Coh} of quaternionic reflection groups (see also \cite{Tay} and \cite{Wal}). Specifically, we show that for the infinite family $W_n(\Gamma,\Delta)$ (see \ref{infinite family} for the definition), the hyperplane complement is an iterated fiber bundle, while for the exceptional groups we compute the structure of the lattice of parabolic subgroups explicitly enough to determine the Poincar\'e polynomial. This calculation should also be of use in questions having to do with the structure of the class of parabolic subgroups, for instance in the study of symplectic reflection algebras. 

\subsection{Numerical invariants from polynomial invariants and the coinvariant ring} Given a complex reflection group $W$ with reflection representation $\hh$, the ring $\CC[\hh]^W$ of $W$-invariant polynomial functions on $\hh$ is generated by $n=\mathrm{dim}(\hh)$ algebraically independent homogeneous polynomials of degrees $d_1 \leq d_2 \leq \cdots \leq d_n$. This sequence $d_1 \leq d_2 \leq \cdots \leq d_n$ is independent of the choice of generators, and the integers $d_1,\dots,d_n$ are called the \emph{degrees} of $W$. The ring $\CC[\hh]^W$ is the ring of functions on the quotient scheme $\hh/W$, and the \emph{coinvariant ring} of $W$ is the ring $\CC[\pi^{-1}(0)]$ of functions on the scheme-theoretic fiber $\pi^{-1}(0)$ of the quotient map $\pi: \hh \to \hh/W$. The structure of the graded $W$-module $\CC[\pi^{-1}(0)]$ provides another family of numerical invariants of $W$: upon forgetting the grading it is isomorphic to the regular representation $\CC W$ and so for each irreducible $\CC W$-module $E$ there is a well-determined sequence $m_1 \leq m_2 \leq \cdots \leq m_p$ of $p=\mathrm{dim}(E)$ non-negative integers known as the \emph{exponents} of $E$ with $E$ occurring in degrees $m_1, m_2,\dots,m_p$ in $\CC[\pi^{-1}(0)]$. In particular, the representation $\hh$ appears in degrees $d_1-1,d_2-1,\dots,d_n-1$ in $\CC[\pi^{-1}(0)]$, and the numbers $e_1=d_1-1,\dots,e_n=d_n-1$ are known simply as the \emph{exponents} of $W$. The \emph{co-exponents} of $W$ are the integers $1=f_1 \leq f_2 \leq \cdots \leq f_n$ so that $\hh^*$ appears in degrees $f_1,f_2,\dots,f_n$ in $\CC[\pi^{-1}(0)]$. If $\hh \cong \hh^*$, or in other words if $W$ is a real reflection group, then the exponents and co-exponents coincide.

\subsection{Exponents and co-exponents via topology of the hyperplane complement and the combinatorics of the intersection lattice} These numerical invariants arise in another context. Namely, given a complex reflection group $W$ we define two polynomials
$$c_W(t)=\sum_{w \in W} t^{\mathrm{codim}(\mathrm{fix}(w))}$$ and
$$p_W(t)=\sum_{d=0}^\infty \mathrm{dim}(H^d(\mathcal{M}_W)) t^d,$$ where $\mathcal{M}_W$ is the hyperplane complement for the group $W$, obtained by removing all its reflecting hyperplanes from the space $V$ on which it acts. These polynomials factor in terms of the exponents and coexponents of $W$:
$$c_W(t)=\prod_{i=1}^n (1+e_i t) \quad \text{and} \quad p_W(t)=\prod_{i=1}^n (1+f_i t).$$ Thus if $W$ is a real reflection group, we have $p_W(t)=c_W(t)$. Both of these definitions still make sense for a quaternionic reflection group, and the purpose of this paper is to point out that while the algebraic definitions of the exponents and coexponents no longer work, the polynomials $c_W(t)$ and (especially!) $p_W(t)$ exhibits fascinating and mysterious behavior very similar to the case of a complex reflection group.  

\subsection{The codimension and Poincar\'e polynomials for a quaternionic reflection group} Given a quaternionic reflection group $W$ with reflection representation $V$, we search for numerical invariants. The paper \cite{CaGr} introduced three variants of the Coxeter number $h$ of an irreducible real reflection group, $g=2N/n,$ $h=(N+N^*)/n$, and $k=2N^*/n$, where $N$ is the number of reflections, $N^*$ is the number of reflecting hyperplanes and $n$ is the rank. Here we compute the ranks of the cohomology groups with coefficients in $\ZZ$ of the  hyperplane complement $V^\circ$ associated with $W$, with the hope of discovering an analog of the coexponents. 

We encode this information in the Poincar\'e polynomial $p_W(t)$, which is the generating function for these ranks:
$$p_W(t)=\sum_{d=0}^\infty \mathrm{rk}\left(H^{3d}(V^\circ,\ZZ) \right) t^d.$$ We note that the cohomology groups are always free $\ZZ$-modules, and are zero in degrees not divisible by $3$ (essentially because the $3$-sphere $S^3$ has these properties), so $p_W(t)$ contains all the information about the cohomology groups up to isomorphism. We also define the codimension generating function $c_W(t)$ just as in the complex case (but taking $\HH$-codimension).

Our main theorem is the computation of these polynomials for all quaternionic reflection groups. Since they are multiplicative, $p_{W_1 \times W_2}=p_{W_1} p_{W_2}$ and $c_{W_1 \times W_2}=c_{W_1} c_{W_2}$, it suffices to compute them for irreducible groups, which we do using Cohen's \cite{Coh} classification. We may assume our group $W$ is not complex.

\newpage

\begin{theorem} \label{main theorem}

\begin{itemize} 
\item[(a)] For a quaternionic group $W$ of rank $2$ containing $N$ reflections and $N^*$ reflecting hyperplanes, $$p_W(t)=(1+t)(1+(N^*-1)t) \quad \text{and} \quad c_W(t)=1+Nt+(|W|-N-1)t^2.$$ 

\item[(b)] For a positive integer $n \geq 3$, a finite non-abelian subgroup $\Gamma$ of $\HH^\times$ with order $|\Gamma|=m$, and a normal subgroup $\Delta$ of $\Gamma$ such that $\Gamma/\Delta$ is abelian, the Poincar\'e polynomial $p_W(t)$ of $W_n(\Gamma,\Delta)$ is
$$p_W(t)=\left(1+t \right) \left(1+(1+m)t \right) \left(1+(1+2m)t \right)\cdots\left(1+(1+(n-1)m)t \right),$$ and the polynomial $c_W(t)$ is
$$c_W(t)=\left(1+(m-1)t \right)\left(1+(2m-1)t \right) \cdots \left(1+((n-1)m-1)t \right) \left(1+(n p-1)t \right),$$ where $p=|\Delta|$ is the order of $\Delta$. 
\item[(c)] For the exceptional irreducible quaternionic groups of rank at least $3$ the Poincar\'e polynomials and codimension generating functions are given in the tables below.
\end{itemize}
\end{theorem}

\begin{table}[h]
\caption{Poincar\'e polynomials for the exceptional groups of rank at least $3$}
\centering
\begin{tabular}{l c c rrrrrrr}
\hline \hline
Group $W$  &\multicolumn{1}{c}{Poincar\'e polynomial $p_W(t)$}
\\ [0.5ex]
\hline 
&$1+63t+987t^2+925t^3$ \\[-1ex]
\raisebox{1.5ex}{$W(Q)$} &   $=(1+t)(1+25t)(1+37t)$ \\[1ex]

&$1+315t+23667t^2+23353t^3$ \\[-1ex]
\raisebox{1.5ex}{$W(R)$} &   $=(1+t)(1+121t)(1+193t)$ \\[1ex]

&$1+36t+438t^2+1924t^3+1521t^4$ \\[-1ex]
\raisebox{1.5ex}{$W(S_1)$} &   $=(1+t)(1+9t)(1+13t)(1+13t)$ \\[1ex]

&$1+72t+1722t^2+14176t^3+12525t^4$ \\[-1ex]
\raisebox{1.5ex}{$W(S_2)$} &   $=(1+t)(1+25t)(1+46t+501t^2)$ \\[1ex]

&$1+180t+10326t^2+195220t^3+185073t^4$ \\[-1ex]
\raisebox{1.5ex}{$W(S_3)$} &   $=(1+t)(1+49t)(1+130t+3777t^2)$ \\[1ex]

&$1+180t+10614t^2+207892t^3+197457t^4$ \\[-1ex]
\raisebox{1.5ex}{$W(T)$} &   $=(1+t)(1+61t)(1+118t+3237t^2)$ \\[1ex]

& $1+165t+10010t^2+265210 t^3+2657589t^4+2402225t^5$ \\[-1ex]
\raisebox{1.5ex}{$W(U)$} &   $=(1+t)(1+25t)(1+37t)(1+49t)(1+53t)$ \\[1ex]

\hline 
\end{tabular}
\label{tab:PPer}
\end{table}

\begin{table}[h]
\caption{Codimension generating functions for the exceptional groups of rank at least $3$}
\centering
\begin{tabular}{l c c rrrrrrr}
\hline \hline
Group $W$  &\multicolumn{1}{c}{Codimension generating function $c_W(t)$}
\\ [0.5ex]
\hline 
&$1+63t+1239t^2+10793t^3$ \\[-1ex]
\raisebox{1.5ex}{$W(Q)$} & Irreducible    \\[1ex]

&$1+315t+27447t^2+1181837t^3$ \\[-1ex]
\raisebox{1.5ex}{$W(R)$} & Irreducible   \\[1ex]

&$1+36t+438t^2+2180t^3+4257t^4$ \\[-1ex]
\raisebox{1.5ex}{$W(S_1)$} &   $=(1+11t)(1+25t+163t^2+387t^3)$ \\[1ex]

&$1+72t+1722t^2+16496t^3+64653t^4$ \\[-1ex]
\raisebox{1.5ex}{$W(S_2)$} &   $=(1+23t)(1+49t+595t^2+2811t^3)$ \\[1ex]

&$1+180t+10974t^2+272420t^3+3034185t^4$ \\[-1ex]
\raisebox{1.5ex}{$W(S_3)$} &   $=(1+71t)(1+109t+3235t^2+42735t^3)$ \\[1ex]

&$1+180t+10614t^2+244628t^3+2336577t^4$ \\[-1ex]
\raisebox{1.5ex}{$W(T)$} &   $=(1+59t)(1+121t+3475t^2+39603t^3)$ \\[1ex]

& $1+165t+10010t^2+279290t^3+3658149t^4+23423905t^5$ \\[-1ex]
\raisebox{1.5ex}{$W(U)$} &   Irreducible \\[1ex]
\hline 
\end{tabular}
\label{tab:PPer}
\end{table}

\newpage

The polynomials $p_W(t)$ factor into irreducible factors with positive integer coefficients, all of which are linear except for a single quadratic factor for the three exceptional groups of type $S_2$, $S_3$, and $T$. The factor $(1+t)$ that appears in all of them occurs simply because these arrangements are central, but the other linear factors are more mysterious. It would be very interesting to explain these phenomena via some structural property shared by all quaternionic reflection groups. The polynomials $c_W(t)$ also factor into irreducible factors with positive integer coefficients, and share (to a lesser extent) this tendency to have linear factors, but this is perhaps less surprising (a partial analog of the phenomenon in \ref{c product formula} could explain it).

Our proof of Theorem \ref{main theorem} occupies the remainder of the paper. Part (a) follows immediately from the definitions and the material in the next section. We prove parts (b) and (c) in a case by case fashion. By Cohen's classification theorem \cite{Coh}, every irreducible quaternionic reflection group is either a complex reflection group or covered by our theorem. We have recently learned that the graded dimension of the $W$-invariant part of the cohomology has also been calculated by Giordani-Röhrle-Schmitt \cite{GRS}.

\section{Topological preliminaries: the cohomology of a quaternionic hyperplane complement}

\subsection{Summary} The purpose of this section is to collect the basic facts about hyperplane complements we will use, in the form we will use them. Most of this has appeared in \cite{Sch}. For us the important point is that, just as for complex arrangements, the Poincar\'e polynomial of the complement of a quaternionic hyperplane arrangement may be computed in terms of the M\"obius function of the intersection lattice of the arrangement. The proof of this is actually easier than the corresponding fact for complex arrangements: the splitting of the long exact sequence for a complex arrangement equipped with a distinguished hyperplane is automatic for dimension reasons in the quaternionic case, as we explain below. In fact, it seems likely that one could turn this around, give an \emph{a priori} proof that the cohomology of the quaternionification of a complex hyperplane complement is isomorphic (up to changing the grading) to the cohomology of the complex complement, and thereby obtain a more efficient proof of the results in the complex case by use of the quaternionic case.

\subsection{Setup} We write $\HH$ for the division ring of quaternions, which is the $\RR$-vector space with basis $1,i,j,k$ and $\RR$-linear multiplication determined by
$$i^2=j^2=k^2=ijk=-1.$$ We write $\overline{h}=a-bi-cj-dk$ for the conjugate of a quaternion $h=a+bi+cj+dk \in \HH$ with $a,b,c,d \in \RR$, and $\HH^\times$ for the group of non-zero quaternions.

We fix a finite-dimensional right $\HH$-vector space $V$. Let $\A$ be a finite collection of linear (that is, containing $0$) hyperplanes in $V$ (sometimes called a \emph{central} hyperplane arrangement), and let $\cM_\A$ be the corresponding hyperplane complement $$\cM_\A=V \setminus \bigcup_{H \in \A} H.$$ We are interested in the cohomology with integer coefficients of $\cM_\A$. It turns out that this is entirely controlled by the structure of the poset $L(\A)$ whose elements are the finite intersections $X=H_1 \cap \cdots \cap H_m$ of elements of $\A$, with the order $X \geq Y$ if $X \subseteq Y$. Thus $V$ (the empty intersection) is the unique minimal element of $L(\A)$, and the intersection of all the elements of $\A$ is the unique maximal element (often $\{ 0\}$) of $L(\A)$. This poset is ranked by codimension: we define $\mathrm{rk}(X)=\mathrm{codim}_V(X)$ for $X \in L(\A)$. 

\subsection{The inductive approach} Fix $H_0 \in \A$. We obtain two new arrangements: first, the arrangement $\A_0$ in the vector space $H_0$ consists of the set $H_0 \cap H$ of hyperplanes in $H_0$ obtained by intersecting $H_0$ with those $H \in \A$ such that $H_0$ is not contained in $H$. Second, the arrangement $\A^0=\A \setminus \{H_0 \}$ is obtained simply by omitting $H_0$ from $\A$. We write $\cM_0$ for the complement in $H_0$ of the arrangement $\A_0$ and $\cM^0$ for the complement in $V$ of the arrangement $\A^0$. Thus $\cM_\A$ is an open subset of $\cM^0$ with complement equal to $\cM_0$ (which is therefore a closed subset of $\cM^0$). We will use this configuration to study the arrangement $\A$ and its complement $\cM_\A$ by induction on the cardinality of $\A$.

\subsection{Cohomology} We will use only cohomology with integer coefficients. Thus for a topological space $X$ we write simply $H^d(X)=H^d(X;\ZZ)$ for the $d$th cohomology group with integer coefficients. For an abelian group $A$, we write $\mathrm{rk}(A)$ for its rank (which is a non-negative integer or $\infty$). Given a quaternionic hyperplane arrangement $\A$, we define the \emph{Poincar\'e polynomial} $p_\A(t)$ by
$$p_\A(t)=\sum_d \mathrm{rk}(H^{3d}(\cM_\A))t^d.$$ As we will see below, the cohomology groups of $\cM_\A$ are finitely generated free abelian groups and are zero except in degrees divisible by $3$, so $p_\A(t)$ contains all the information about their isomorphism classes.

\subsection{The topological side of the induction} The key observation for the topological side of this inductive approach is that $\cM_0$ has a tubular neighborhood $U$ in $\cM^0$, so that the pair $(U,U \setminus \cM_0)$ is homeomorphic to $(\cM_0 \times \HH,\cM_0 \times \HH^\times)$. Thus by excision, the K\"unneth formula, and the fact that the relative cohomology of the pair $H^d(\HH,\HH^\times)$ is zero for $d \neq 0$ and isomorphic to $\ZZ$ for $d=4$,

$$H^d(\cM^0,\cM_\A) \cong H^d(U,U \setminus \cM_0)\cong H^d(\cM_0 \times \HH,\cM_0 \times \HH^\times) \cong H^{d-4}(\cM_0) \otimes H^4(\HH,\HH^\times) \cong H^{d-4}(\cM_0).$$ 

The long exact sequence for the pair $(\cM^0,\cM_\A)$ is
$$H^d(\cM^0,\cM_\A) \to H^d(\cM^0) \to H^d(\cM_\A) \to H^{d+1}(\cM^0,\cM_\A),$$ which by the preceding isomorphism produces an exact sequence
$$H^{d-4}(\cM_0) \to H^d(\cM^0) \to H^d(\cM_\A) \to H^{d-3}(\cM_0).$$ If $\A$ contains only one hyperplane, then $\cM_\A$ is homotopic to $S^3$ and the cohomology is non-zero only in degrees $0$ and $3$. By using the exact sequence above and induction on $|\A|$, it follows that $H^d(\cM_\A)$ is non-zero only for $d$ divisible by $3$. Moreover, when $d$ is divisible by $3$, the integer $d-4$ is not, and it follows that we have an exact sequence
$$0 \to H^d(\cM^0) \to H^d(\cM_\A) \to H^{d-3}(\cM_0) \to 0.$$

Now induction shows furthermore that $H^d(\cM_\A)$ is a free $\ZZ$-module of finite rank, isomorphic to the direct sum $H^d(\cM^0) \oplus H^{d-3}(\cM_0)$ and so the Poincar\'e polynomial $p_\A(t)$ satisfies the recursion
\begin{equation} \label{p recursion}
p_\A(t)=p_{\A^0}(t)+tp_{\A_0}(t).
\end{equation}

\subsection{The combinatorial side of the induction} 

Here we simply restate the results from chapter 2 of \cite{OrTe}, the proofs of which go through without change for arrangements in a vector space over a division ring. We recall that the \emph{M\"obius function} of the intersection lattice $L(\A)$ is the function $\mu: \A \times \A \to \ZZ$ determined by the normalization $\mu(X,X)=1$ for all $X \in L(\A)$, $\mu(X,Y)=0$ unless $X \leq Y$, and for $X<Y$
$$0=\sum_{X \leq Z \leq Y} \mu(X,Z).$$ 

Just as in \cite{OrTe}, chapter 2, section 2.2, we write $\mu(X)=\mu(V,X)$ for $X \in L(\A)$. We then have $(-1)^{\mathrm{rk}(X)} \mu(X) >0$ for all $X \in L(\A)$, and defining the polynomial $\pi_\A(t)=\sum (-1)^{\mathrm{rk}(X)} \mu(X) t^{\mathrm{rk}(X)}$ we obtain the recursion
$$\pi_A(t)=\pi_{\A^0}(t)+t\pi(\A_0,t),$$ which together with the initial condition $\pi_\emptyset(t)=1$ shows that $\pi_\A(t)=p_\A(t)$. Thus
\begin{equation} \label{L formula}
\mathrm{rk}(H^{3d}(\cM_\A))=\sum_{\substack{X \in L(\A) \\ \mathrm{rk}(X)=d}} (-1)^{\mathrm{rk}(X)} \mu(X).
\end{equation} 
\subsection{The reformulation in terms of line systems}

Each hyperplane arrangement $\A$ in an $\HH$-vector space $V$ equipped with a positive definite Hermitian form produces a \emph{line system} $L$: we take $L$ to be the set of lines consisting of the lines in $V$ that are orthogonal to the hyperplanes in $\A$. The subspace lattice $S(L)$ generated by $L$ then consists of the orthogonal complements of the elements of the intersection lattice $L(\A)$. The elements of $S(L)$ are the \emph{flats} of $L$, and we refer to the elements of $S(L)$ of dimension $d$ as the \emph{$d$-flats} of $L$. 

Given an element $X \in S(L)$ we define $\mu(X)=(-1)^{\mathrm{rk}(X^\perp)} \mu(X^\perp)$. Thus \eqref{L formula} may be written
\begin{equation} \label{line system cohomology formula}
\mathrm{rk}(H^{3d}(\cM_\A))=\sum_{\substack{X \in S(L) \\ \mathrm{dim}(X)=d}} \mu(X).
\end{equation} This is the form in which we will compute the Poincar\'e polynomials for the line systems associated with the exceptional quaternionic reflection groups. 

The definition of the M\"obius function implies the following recursion for computing $\mu(X)$ for $X$ in $S(L)$: $\mu(0)=1$ and
\begin{equation} \label{mu recursion} 
\mu(X)=\sum_{\substack{Y \in S(L) \\ Y \subsetneqq X}} (-1)^{\mathrm{dim}(X)-\mathrm{dim}(Y)-1} \mu(Y),
\end{equation} where $\mathrm{dim}(X)$ is the $\HH$-dimension of the subspace $X$.

\subsection{A map to $(\HH^\times)^\A$} We will not use the remaining assertions in the body of the paper, but we include them here because they fit well with the theme. For each hyperplane $H \in \A$ we choose a linear form $\alpha_H:V \to \HH$. Together these define a map $V \to \HH^\A$ sending $v \in V$ to the point of $\HH^\A$ with coordinates $\alpha_H(v)$ for $H \in \A$, restricting to a map $\cM_\A \to (\HH^\times)^\A$ which commutes with the right $\HH^\times$-action on both. Now a space with an $\HH^\times$-action inherits a degree $-3$ differential on its cohomology algebra. In the case of $(\HH^\times)^\A$, its cohomology algebra is an exterior algebra on degree three variables $\omega_H$ in bijection with the hyperplanes $H \in \A$, and the derivation $d$ is the unique one sending $\omega_H$ to $1$ for all $H \in \A$.

The map $\cM_\A \to (\HH^\times)^\A$ induces a map on cohomology $H^*((\HH^\times)^\A) \to H^*(\cM_\A)$ that commutes with the differentials induced by the $\HH^\times$-actions.
 Moreover one checks that for any dependent set $H_1,\dots,H_p$ in $\A$, the element $\omega_{H_1} \cdots \omega_{H_p}$ is in the kernel of this map on cohomology. Since this map commutes with differentials, its kernel contains $\partial(\omega_{H_1} \cdots \omega_{H_p})$ for any dependent set $H_1,\dots,H_p$ in $\A$, where $\partial$ is the derivation induced by the $\HH^\times$-action.

\subsection{The Orlik-Solomon algebra and the algebraic side of the induction} 

 The \emph{Orlik-Solomon algebra} $A(\A)$ is the quotient of the exterior $\ZZ$-algebra on degree $3$ generators $\omega_H$ for all $H \in \A$ by the ideal generated by all $\partial(\omega_{H_1} \cdots \omega_{H_p})$ for dependent sets $H_1,\dots,H_p$ in $\A$. The defining ideal of $A(\A)$ is homogeneous, so $A(\A)$ is a graded algebra with graded pieces that we write as $A(\A)^d$, and there are exact sequences
 $$0 \to A(\A^0)^d \to A(\A)^d \to A(\A_0)^{d-3} \to 0$$ of $\ZZ$-modules. Using these and induction on $|\A|$ one proves that the natural maps $A(\A)^d \to H^d(\cM_\A)$ are isomorphisms. 

\subsection{The graded trace of an automorphism of $\A$} 

If $g$ is an automorphism of $V$ that preserves the set $\A$ of hyperplanes, then $g$ induces automorphisms of the cohomology ring of $\cM_\A$ and on the Orlik-Solomon algebra $A(\A)$, and the isomorphism $A(\A) \to H^*(\cM_\A)$ constructed above is $g$-equivariant. Just as in chapter 6 of \cite{OrTe} we obtain the graded trace of $g$ on $H^*(\cM_\A)$:
$$\mathrm{tr}(g,H^{3d}(\cM_\A))=(-1)^d \sum_{\substack{X \in L(\A)^g \\ \mathrm{rk}(X)=d}} \mu^g(V,X),$$ where $L(\A)^g$ is the poset of flats that are fixed (setwise) by $g$ and $\mu^g$ is its M\"obius function.  

\section{Reflection groups and line systems}

\subsection{Hermitian forms and angles between lines} 

Given a right $\HH$-vector space $V$, a \emph{Hermitian form} on $V$ is a function $(\cdot,\cdot):V \times V \to \HH$ such that
$$(v_1,v_2)=\overline{(v_2,v_1)} \quad \text{and} \quad (v_1,v_2a+v_3b)=(v_1,v_2)a+(v_1,v_2)b \quad \hbox{for all $v_1,v_2,v_3 \in V$ and all $a,b \in \HH$.}$$ Thus $(\cdot,\cdot)$ is $\HH$-linear in the second variable and satisfies
$$(av_1+bv_2,v_3)=\overline{a} (v_1,v_3)+\overline{b} (v_2,v_3) \quad \hbox{for all $a,b \in D$ and all $v_1,v_2,v_3 \in V$.}$$ Evidently $(v,v) \in \RR$ for all $v \in V$, and we say that the form is \emph{positive definite} if $(v,v)>0$ for all $v \in V \setminus \{ 0 \}$. 

For example, the form on $\HH^n$ given by the formula
$$((x_1,\dots,x_n),(y_1,\dots,y_n))=\sum_{i=1}^n \overline{x_i} y_i$$ is a positive definite Hermitian form. We will assume throughout that $V$ is equipped with a fixed positive definite Hermitian form. 

The \emph{Cauchy-Schwarz inequality} is
$$|(v_1,v_2)| \leq |v_1| | v_2| \quad \hbox{for all $v_1,v_2 \in V$, with equality if and only if $v_1 \HH=v_2 \HH$.}$$ This holds for any positive definite Hermitian form, and allows us to define the \emph{angle} $\alpha$ between two one-dimensional spaces $\ell$ and $\ell'$ of $V$ by 
\begin{equation} \label{angle definition}
\mathrm{cos}(\alpha)=\frac{|(v_1,v_2)|}{|v_1| |v_2|}
\end{equation} for any non-zero $v_1 \in \ell_1$ and $v_2 \in \ell_2$. 

\subsection{The general linear and unitary groups} Let $V$ be a finite-dimensional right $\HH$-vector space equipped with a positive definite Hermitian form $(\cdot,\cdot)$. The \emph{general linear group} of $V$ is the group $\mathrm{GL}(V)$ of all invertible $\HH$-linear transformations $g:V \to V$. It acts on $V$ on the left, and if the dimension of $V$ is $n$ then upon choosing a basis of $V$ may be identified with the group of all invertible $n$ by $n$ matrices with entries in $\HH$, acting on $\HH^n$ by left multiplication.  The \emph{unitary group} of $V$ is 
$$U(V)=\{g \in \mathrm{GL}(V) \ | \ (g(v_1),g(v_2))=(v_1,v_2) \ \hbox{for all $v_1,v_2 \in V$} \}.$$

\subsection{Reflection groups} Let $V$ be a right $\HH$-vector space equipped with a positive definite Hermitian form $(\cdot,\cdot)$, and let $W \subseteq U(V)$ be a finite group of unitary $\HH$-linear transformations. A \emph{reflection} in $W$ is an element $r \in W$ such that
$$\mathrm{codim}_\HH(\mathrm{fix}_V(r))=1.$$ We write $R$ for the set of all the reflections in $W$, and we say that $W$ is a \emph{reflection group} if it is generated by $R$. Given a reflection group $W$, we let
$$\mathcal{A}_W=\{ \mathrm{fix}(r) \ | \ r \in R\}$$ be the reflection arrangement for $W$, and we let
$$\mathcal{M}_W=V \setminus \bigcup_{H \in \mathcal{A}_W} H$$ be the corresponding (quaternionic) hyperplane complement. We also obtain a line system $L$ associated with $W$ whose elements are the orthogonal complements of the reflecting hyperplanes.

\subsection{The definition of $W_n(\Gamma,\Delta)$} \label{infinite family} Here we define the infinite family of reflection groups that accounts for all but $7$ of the irreducible quaternionic reflection groups in rank at least $3$. We fix a finite subgroup $\Gamma \leq \HH^\times$ of the quaternions together with a normal subgroup $\Delta \leq \Gamma$ such that the quotient group $\Gamma/\Delta$ is abelian. If $\Gamma$ is cyclic then every subgroup $\Delta$ of $\Gamma$ satisfies these properties. It is often most convenient to specify $\Delta$ by specifying the character group of $\Gamma/\Delta$, which is a subgroup of the group of linear characters of $\Gamma$. 

The group $W_n(\Gamma,\Delta)$ is the subgroup of $\mathrm{GL}_n(\HH)$ generated by the group $S_n$ of permutation matrices together with those diagonal matrices $\mathrm{diag}(\gamma_1,\gamma_2,\dots,\gamma_n)$ such that $\gamma_1,\dots,\gamma_n \in \Gamma$ and $\gamma_1 \gamma_2 \cdots \gamma_n \in \Delta$. When $\Delta=\Gamma$ we abbreviate by writing simply $W_n(\Gamma)=W_n(\Gamma,\Delta)$. For readers more familiar with complex reflection groups, we note that it $\Gamma$ is cyclic of order $\ell$ and $\Delta$ is the subgroup of $\Gamma$ of order $\ell/m$ for some divisor $m$ of $\ell$, then the group we denote $W_n(\Gamma,\Delta)$ is often written as $G(\ell,m,n)=W_n(\Gamma,\Delta)$. 

We will use the following notation: given $\gamma \in \Gamma$ and $1 \leq m \leq n$, we write $\gamma_m$ for the diagonal matrix with $\gamma$ in the $m$th position and $1$'s elsewhere. We will also identify $W_{n-1}(\Gamma)$ with the subgroup of $W_n(\Gamma)$ fixing the vector $e_n=(0,0,\dots,0,1)$. 

\subsection{Parabolic subgroups} Let $W$ be a reflection group with reflection representation $V$. A subgroup $W'$ of $W$ is \emph{parabolic} if there exists $v \in V$ such that $W'=W_v$ is the stabilizer of $v$ in $W$. It is a consequence of the work \cite{BST}of Bellamy-Schmitt-Thiel that each parabolic subgroup $W'$ of $W$ is itself a reflection group. 

Given a parabolic subgroup $W'$ of $W$, we may decompose $V$ as $V^{W'} \oplus V'$, where $V'$ is the orthogonal complement to the fixed space $V^{W'}$. We refer to $V'$ as the \emph{support} of $W'$, and the \emph{rank} of $W'$ is by definition the dimension of $V'$. In particular the rank of $W$ is the dimension of $V$ minus the dimension of $V^W$ (so it is equal to the dimension of $V$ if $W$ is non-trivial and $V$ is an irreducible representation of $W$). 

The support of $W'$ may also be described as follows: for each $H \in \mathcal{A}_W$ containing $V^{W'}$ we let $\ell_H$ be its orthogonal complement. The span of all of these lines is the support of $W'$. In this way parabolic subgroups of $W$ are in bijection with elements $X$ of the subspace lattice $S(L)$ of the line system of $W$.

\subsection{The numerical invariants $\mu$ and $e$} Given a quaternionic reflection group $W$ with reflection representation $V$, we define an element $w \in W$ to be \emph{elliptic} if $\mathrm{fix}_V(w)=\{ 0 \}$. We then define $e(W)$ to be the number of elliptic elements of $W$. Evidently we have
$$e(W)=|W|-\sum e(W'),$$ where the sum runs over all proper parabolic subgroups $W'$ of $W$, each of which we regard as acting on its support in $V$. It is also an immediate consequence of the definitions that $$e(W_1 \times W_2)=e(W_1) e(W_2).$$

The invariant $\mu(W)$ is defined to be the dimension of the top cohomology group of the hyperplane complement $\mathcal{M}_W$ of $W$. More specifically, if $W$ is a group of rank $n$ then we define
$$\mu(W)=\mathrm{dim}(H^{3n}(\mathcal{M}_W)).$$ Just as for $e(W)$ there is a recurrence determining $\mu(W)$:
$$\mu(W)=\sum (-1)^{\mathrm{rk}(W)-\mathrm{rk}(W')-1} \mu(W'),$$ where the sum runs over all parabolic subgroups $W'$ of $W$ (by using \eqref{line system cohomology formula} and the recursion for the M\"obius function). And just as for $e(W)$ the invariant $\mu(W)$ is multiplicative (by the K\"unneth formula):
$$\mu(W_1 \times W_2)=\mu(W_1) \times \mu(W_2).$$

\subsection{The codimension generating function}

We consider the generating function $$c_W(t)=\sum_{w \in W} t^{\mathrm{codim}(\mathrm{fix}(w))}.$$ Evidently only the identity has fixed space of codimension $0$. At the other extreme, if $n=\mathrm{dim}(V)$ then the coefficient of $t^n$ in $c_W(t)$ is the number $e(W)$ of elliptic elements of $W$. By \cite{BST} the fixed space of each $w \in W$ is an element $X$ of the intersection lattice: each $w \in W$ can be written as a product of reflections in (some of) the hyperplanes of the arrangement of $W$ containing its fixed space, and it follows that its fixed space is precisely the intersection of those hyperplanes of the arrangement of $W$ that contain it. 

The elements with fixed space of codimension $1$ are the reflections, and those with fixed space of codimension $2$ are the elements of a rank two parabolic subgroup that are not reflections or the identity; in other words, they are the elliptic elements of some rank two parabolic subgroup of $W$. In general, the number of elements of $W$ with fixed space of codimension $d$ is the total number of elliptic elements of rank $d$ parabolic subgroups of $W$. Thus
\begin{equation} \label{c e formula}
c_W(t)=\sum_d \left(\sum_{\mathrm{rk}(W')=d} e(W') \right)t^d,
\end{equation} where the inner sum runs over all rank $d$ parabolic subgroups $W'$ of $W$.

\subsection{Lines and angles} As above, we let $V$ be a finite-dimensional right $D$-vector space equipped with a positive definite Hermitian form $(\cdot,\cdot)$. A \emph{line} $\ell$ in $V$ is a one-dimensional subspace $\ell \leq V$. Given two lines $\ell$ and $k$, spanned by vectors $v_1$ and $v_2$, as above the angle between them is the real number $\alpha$ defined by
$$\cos(\alpha)=\frac{|(v_1,v_2)|}{|v_1| |v_2|} \quad \hbox{and $0 \leq \alpha \leq \pi/2$.}$$ 

\begin{lemma} \label{star uniqueness}
Suppose $k$ and $\ell$ are lines and the angle between them is $\pi/3$. Then there is a unique line $m$ in the plane spanned by $k$ and $\ell$ that is at angle $\pi/3$ to both $k$ and $\ell$.
\end{lemma}
\begin{proof}
Choose basis elements $v_1$ and $v_2$ for $k$ and $\ell$, such that $(v_1,v_1)=2=(v_2,v_2)$. Existence: by scaling $v_1$ we may assume $(v_1,v_2)=-1$. Put $v=-v_1-v_2$. Then 
$$(v,v)=(v_1,v_1)+2(v_1,v_2)+(v_2,v_2)=2-2+2=2 \quad \text{and} \quad (v,v_1)=-(v_1,v_1)-(v_2,v_1)=-2+1=-1=(v,v_2).$$ It follows that the line spanned by $v$ is at angle $\pi/3$ to $k$ and $\ell$. 

Now we prove uniqueness. Suppose $m$ is a line at angle $\pi/3$ to each of $k$ and $\ell$, and choose a basis element $v$ of $m$ such that $(v,v)=2$. By scaling, we may assume 
 $$(v_1,v_2)=-1 \quad \text{and} \quad (v_1,v)=-1.$$ Since the angle between the lines spanned by $v_2$ and $v$ is $\pi/3$ we furthermore have 
 $$|(v_2,v)|=1.$$ Write $v=v_1 a+v_2 b$ for $a,b \in \HH$. We obtain 
$$-1=(v_1,v)=(v_1,v_1 a+ v_2 b)=2a-b \quad \implies \quad b=2a+1$$ and
$$2=(v,v)=(v,v_1)a+(v,v_2)b=-a+(v,v_2)b.$$ Now 
$$(v_2,v)=-a+2b \quad \implies \quad (v,v_2)=-\overline{a}+2\overline{b} \quad \implies \quad 2=-a+(-\overline{a}+2\overline{b})b.$$ Substituting $b=2a+1$ into this gives
$$2=-a+(-\overline{a}+2 (\overline{2a+1}))(2a+1)=-a-\overline{a}-2|a|^2+2(2\overline{a}+1)(2a+1).$$ Simplifying this implies
$$a+\overline{a}=-2|a|^2.$$ 

But also
$$1=|(v_2,v)|^2=|-a+2b|^2=|-a+2(2a+1)|^2=|3a+2|^2.$$ Thus if $a=x+yi+zj+wk$ for real numbers $x,y,z,w$ then $x=\frac{a+\overline{a}}{2}=-|a|^2$ and
$$1=(3x+2)^2+9y^2+9z^2+9w^2=9(x^2+y^2+z^2+w^2)+12x+4=9|a|^2-12|a|^2+4.$$ It follows that $|a|=1$, and hence $x=-1$ and finally $a=-1$, $b=2a+1=-1$. 
\end{proof}

\subsection{Hyperplanes, and reflections of order $2$}  A \emph{hyperplane} $H$ in $V$ is a codimension-one subspace $H \leq V$. Our choice of form induces a bijection between lines and hyperplanes, given by the rule
$$\ell^\perp=\{v \in V \ | \ (v,x)=0 \quad \hbox{for all $x \in \ell$} \}.$$ Given a hyperplane $H$ in $V$, we choose a non-zero vector $\alpha$ in the line $H^\perp$ orthogonal to $H$ and define the \emph{reflection of order $2$} in $H$ by the formula
$$r_H(v)=v-\alpha \frac{2}{(\alpha,\alpha)} (\alpha,v).$$ Given a line $\ell$ we will also write $r_\ell$ for the reflection of order $2$ in the hyperplane $H=\ell^\perp$. The first part of the next lemma justifies omitting the choice of $\alpha$ from the notation for $r_H$:
\begin{lemma} \label{reflections lemma}
\begin{enumerate}
\item[(a)] The reflection $r_H$ is independent of the choice of $\alpha$.
\item[(b)] We have $(r_H(v_1),r_H(v_2))=(v_1,v_2)$ for all $v_1,v_2 \in V$.
\item[(c)] We have $r_H(v)=v$ for all $v \in H$.
\item[(d)] We have $r_H(v)=-v$ for all $v \in H^\perp$. 
\item[(e)] If $\ell$ and $k$ are lines at angle $\pi/3$ to one another, then $r_\ell(k)$ is the unique line in their span at angle $\pi/3$ to both.
\end{enumerate}
\end{lemma} 
\begin{proof}
This is all straightforward. We prove only (e). We may choose $v_1 \in \ell$ and $v_2 \in k$ with $(v_1,v_2)=-1$ and $(v_1,v_1)=2=(v_2,v_2)$. Then by definition
$$r_\ell(v_2)=v_2-v_1(v_1,v_2)=v_1+v_2,$$ so that $r_\ell(k)$ is at angle $\pi/3$ to both $\ell$ and $k$. By Lemma \ref{star uniqueness} we have proved (e). 
\end{proof} 

The order of the product $r_\ell r_{\ell'}$ of two reflections of order two is determined by the angle between $\ell$ and $\ell'$; the following lemma is a straightforward check.
\begin{lemma} Let $\ell$ and $\ell'$ be lines in $V$.
\begin{itemize}
\item[(a)] $r_\ell$ and $r_{\ell'}$ commute if and only if $\ell \perp \ell'$.
\item[(b)] The order of $r_\ell r_{\ell'}$ is $3$ if and only if the angle between $\ell$ and $\ell'$ is $\pi/3$.
\item[(c)] The order of $r_\ell r_{\ell'}$ is $4$ if and only if the angle between $\ell$ and $\ell'$ is $\pi/4$.
\item[(d)] The order of $r_\ell r_{\ell'}$ is $5$ if and only if the angle between $\ell$ and $\ell'$ is either $\pi/5$ or $2 \pi /5$. 
\end{itemize}
\end{lemma}

\subsection{Line systems for reflection groups}

Given a reflection group $W$ with the property that each reflection has order $2$, we may recover $W$ from its line system: $W$ is the group generated by the reflection $r_\ell$ as $\ell$ ranges over all lines perpendicular to a reflecting hyperplane of $W$. Conversely, given a star-closed line system $L$ in $V$, the group generated by the reflections $r_\ell$ is a finite reflection group $W$. It turns out that for the line system $L$ of one of the exceptional groups of rank at least $3$ specified by Cohen \cite{Coh}, every reflection in the resulting group is the reflection with respect to some line in $L$. But this is not immediately obvious and we do not know a proof that does not use the classification and a computer check.

\subsection{Star-closed line systems}

Let $L$ be a finite set of lines in $V$ (via the correspondence $\ell \mapsto \ell^\perp$, specifying $L$ is the same thing as specifying a finite collection of hyperplanes in $V$). Then $L$ is \emph{star-closed} if for all pairs $\ell,\ell' \in L$, we have 
$$r_\ell(\ell') \in L.$$ Assuming $\ell \neq \ell'$, the lines $\ell$ and $\ell'$ span a plane (a $2$-dimensional subspace of $V$). By the formula for $r_\ell$, the line $r_\ell(\ell')$ lies in this plane. 

\subsection{$k$-systems} Let $L$ be a star-closed line system in $V$ that contains at least $2$ lines and let $k \geq 2$ be an integer. We say that $L$ is a $k$-system if the product $r_\ell r_{\ell'}$ is of order at most $k$ for all $\ell,\ell' \in L$, and there exist some $\ell,\ell' \in L$ so that the order of $r_\ell r_{\ell'}$ is exactly $k$. It is straightforward to check that $L$ is a $2$-system if and only if all the lines in $L$ are orthogonal to one another, and that $L$ is a $3$ system if and only if each pair of lines in $L$ is either orthogonal to one another or at angle $\pi/3$ to one another, and there is at least one pair at angle $\pi/3$.

\subsection{$3$-systems} Because of part (e) of Lemma \ref{reflections lemma}, the star-closed line systems such that every pair of lines is at angle $\pi/2$ or $\pi/3$ play a special role. Those that arise in real vector spaces are classified by the simply-laced Dynkin diagrams. In $\CC$-vector spaces, one additional infinite family arises: for the groups of type $G(3,3,n)$, the corresponding set of root lines gives a $3$-system. Moreover, there are two additional exceptional groups: $W(K_5)$ and $W(K_6)$ (the groups $G_{33}$ and $G_{34}$ in the Shephard-Todd notation), corresponding to lines systems $K_5$ and $K_6$. Passing to quaternionic line systems produces, according to the classification theorem from Cohen \cite{Coh}, no additional infinite family of $3$-systems and two new exceptional line systems of ranks $4$ and $5$, denoted $S_1$ and $U$ in Cohen's notation.

\subsection{The Goethals-Seidel decomposition for a $3$-star in a $3$-system}

Here we describe the quaternionic analog of some of the material from chapter 7 of \cite{LeTa}, and which is especially useful in our analysis of the flats for the groups of type $S_1$ and $U$. 

Suppose that $L$ is a $3$-system in a right $\HH$-vector space $V$ equipped with a positive definite Hermitian form $(\cdot,\cdot)$, and that $a,b,c \in L$ are lines at angle $\pi/3$ to one another and together spanning a $2$-dimensional subspace of $V$. It follows that there are no other lines of $L$ in this subspace, which we will call a $3$-star of $L$. We define the following subsets of $L$:

\begin{itemize}
\item[(a)] The set $\Delta$ consists of the lines of $L$ that are perpendicular to all three of $a,b,$ and $c$. 
\item[(b)] The set $\Lambda$ consists of the lines of $L \setminus \{a,b,c \}$ that are not perpendicular to any of $a,b,$ and $c$.
\item[(c)] The set $\Gamma_a$ consists of the lines of $L$ that are perpendicular to $a$ but not to $b$ or $c$; the sets $\Gamma_b$ and $\Gamma_c$ are defined analogously.
\end{itemize} It is clear that $L$ is the disjoint union $$L=\{a,b,c \} \cup \Delta \cup \Lambda \cup \Gamma_a \cup \Gamma_b \cup \Gamma_c.$$ Evidently the cardinalities of these sets depend only on the $W(L)$ orbit of the $3$-star $\{a,b,c \}$. But they may well vary across such orbits, as the examples we compute next show. As a first very simple example we consider the line system of type $A_3$, consisting of the six lines spanned by the six vectors

$$a=(1,-1,0,0), \ b=(1,0,-1,0), \ c=(0,1,-1,0), \ d=(1,0,0,-1), \ e=(0,1,0,-1), \ \text{and} \ f=(0,0,1,-1).$$ Taking our $3$-star to be $\{a,b,c \}$ we evidently have $\Delta=\emptyset=\Lambda$, while $\Gamma_a=\{ f \}$, $\Gamma_b=\{ e \}$, and $\Gamma_c=\{d \}$.

Now we show that the GS decomposition determines the $3$-flats of $L$ containing the $3$-star $\{a,b,c\}$.

\begin{lemma} Let $L$ be a $3$-system and let $\{a,b,c\}$ be a $3$-star of $L$.
\item[(a)] The $3$-flat generated by $\{a,b,c\}$ and an element $\ell$ of $\Delta$ is of type $A_2 \times A_1$, and the number of flats of $L$ of type $A_2 \times A_1$ containing $\{a,b,c\}$ is the cardinality of $\Delta$.

\item[(b)] The $3$-flat generated by $\{a,b,c\}$ and an element $\ell$ of $\Lambda$ is of type $G(3,3,3)$, and the number of flats  of $L$ of type $G(3,3,3)$ containing $\{a,b,c\}$ is the cardinality of $\Lambda$ divided by $6$ (in particular, the cardinality of $\Lambda$ is a multiple of $6$). 

\item[(c)] The $3$-flat generated by $\{a,b,c\}$ and an element $\ell$ of $\Gamma_a$ is of type $A_3$, and the number of flats of $L$ of type $A_3$ containing $\{a,b,c\}$ is the cardinality of $\Gamma_a$. Moreover, for each element $d \in \Gamma_a$ there are unique $e \in \Gamma_b$ and $f \in \Gamma_c$ such that the lines in the $3$-flat generated by $d$ and $\{a,b,c\}$ are precisely $a,b,c,d,e,$ and $f$.
\end{lemma}
\begin{proof}
This follows from Cohen's classification: the only reflection groups of rank three containing only reflections of order two and such that the order of a product of two reflections is either $2$ or $3$ are the groups of types $A_1^{\times 3}$, $A_2 \times A_1$, $A_3$, and $G(3,3,3)$. The group of type $A_1^{\times 3}$ doesn't contain a copy of $A_2$, and the other three are distinguished by the angles their lines make with a given subgroup of type $A_2$.
\end{proof}

\section{First examples of $\mu$ and $e$-invariants}

\subsection{Structure of our counting techniques} In order to apply \eqref{line system cohomology formula} and \eqref{mu recursion} to compute the $\mu$-invariant we classify and count the flats in the line systems of the exceptional reflection groups. In our count of the flats of the exceptional reflection groups, we repeatedly employ slight variants the following strategy. Given that we have counted all flats of rank $d$, we count flats of rank $d+1$ by counting, in two ways, pairs $(X,Y)$ such that $X$ is a rank $d$ flat of a given type, $Y$ is a rank $d+1$ flat of a given type, and $X \subseteq Y$. To carry this out, we must know how many flats of a certain type contain a given flat $X$, and know, given a flat $Y$ of a certain type, how many flats of a certain type it contains. We collect some of this information here for a number of small examples that will be useful later on, while at the same time illustrating the strategy in some simple cases.

\subsection{Notation for line systems} We use the usual notation $A_n$ for the line system of the symmetric group $S_{n+1}$ acting on its irreducible $n$-dimensional reflection representation, and similarly use the notation $B_n$ for the line system of the type $B_n$ Weyl group. We recall that in \ref{infinite family} we have defined the complex reflection groups of type $G(\ell,m,n)$. We will also use this notation below.

\subsection{$A_1^{\times 3}$} We consider first the line system $L=A_1 \times A_1 \times A_1$: for this line system $S(L)$ is a boolean algebra on three generators, of cardinality $8$. There are three elements of $S(L)$ of rank two, all isomorphic to $A_1 \times A_1$. It thus follows from the definition that
$$\mu(L)=3-3+1=1$$ and $e(L)=1$, in accordance with the fact that $e(A_1)=\mu(A_1)=1$ and these functions are multiplicative.

\subsection{$A_1 \times A_2$} Next, for the line system $A_1 \times A_2$, the rank one elements of $S(L)$ are four lines, and the rank two elements are three line systems of type $A_1 \times A_1$ and one line system of type $A_2$. It follows that
$$e(L)=\mu(L)=1-4+3+2=2.$$

\subsection{$A_3$} For the line system $A_3$, the rank one elements of $S(L)$ are six lines and the rank two elements are three line systems of type $A_1 \times A_1$ and four line systems of type $A_2$. Thus
$$e(L)=\mu(L)=1-6+3+4\cdot2=6.$$

\subsection{$B_3$} For the line system $B_3$, the rank one elements of $S(L)$ are the nine lines and the rank two elements come in three types. Firstly, there are six subsystems of type $A_1 \times A_1$. Secondly, there are three subsystems of type $B_2$. And finally, there are four subsystems of type $A_2$. Thus
$$e(L)=\mu(L)=1-9+6+3 \cdot 3+4 \cdot 2=15.$$

\subsection{$G(3,3,3)$} For the line system of the group $G(3,3,3)$, the rank one elements of $S(L)$ are the nine lines, on which the group acts transitively. Fixing one of these lines, every other line is at angle $\pi/3$ to it, and the other eight lines are therefore partitioned into four sets of two lines each, forming together with our fixed line a system of type $A_2$. It follows by counting pairs $(\ell,X)$ consisting of a line $\ell$ of $L$ and a $2$-flat $X$ containing it that there are $12$ of these $2$-flats of type $A_2$. Thus
$$\mu(L)=12 \cdot 2-9+1=16.$$ To compute $e(L)$ we observe that $|G(3,3,3)|=54$, while each of the twelve rank two reflection subgroup contributes $2$ elements with fixed space a line, and the $9$ reflections account for the elements with fixed space of dimension $2$. This leaves $$e(L)=54-24-9-1=20$$ elliptic elements in $G(3,3,3)$.

\subsection{$G(4,4,3)$} For the line system of the group $G(4,4,3)$, the rank one elements of $S(L)$ are the twelve lines, on which the group acts transitively. Fixing one of these lines, say the line spanned by $(1,-1,0)$, the remaining lines may be classified according to the angle they make with this one:
\begin{enumerate}
\item[(a)] The line spanned by $(1,1,0)$ is perpendicular to it.
\item[(b)] The lines spanned by $(1,i,0)$ and $(1,-i,0)$ each make an angle $\pi/4$ with it.
\item[(c)] The lines spanned by vectors of the form $(1,0,\zeta)$ and $(0,1,\zeta)$ make angle $\pi/3$ with it.
\end{enumerate} The lines of the first two sorts make up the unique system of type $B_2$ containing the given line. The remaining eight lines are partitioned into four pairs each forming a $2$-flat of type $A_2$ containing the given line. It follows that there are three $2$-flats of type $B_2$ and sixteen $2$-flats of type $A_2$. Thus
$$\mu(L)=1-12+3\cdot 3+16\cdot 2=30.$$ To compute $e(L)$ we observe that $|G(4,4,3)|=96$ so that the above calculations show
$$e(L)=96-1-12-3 \cdot 3-16\cdot 2=42.$$

\subsection{$W_3(Q,\pm 1)$} For the line system of the group $W_3(Q,\{\pm 1\})$, we have two orbits of lines: one consists of the lines spanned by the three vectors $(1,0,0)$, $(0,1,0)$, and $(0,0,1)$, and the other contains the lines spanned by the vectors of the form $(1,p,0)$, $(1,0,p)$, and $(0,1,p)$ for $p \in Q$. There are then $3$ types of $2$-flat: $3$ of type $W_2(Q,\{ \pm 1\})$ with $\mu$-invariant $9$ and $e$-invariant $21$, $64$ of type $A_2$ with $\mu$-invariant $2$, and $24$ of type $A_1 \times A_1$ with $\mu$-invariant $1$. It follows that the $\mu$-invariant of $W_3(Q,\{\pm 1\})$ is $$\mu(L)=153$$ (in agreement with our topological calculation below) and since $|W_3(Q,\pm 1)|=768$, the $e$-invariant is
$$e(L)=768-3\cdot 21-64\cdot 2-24-27-1=525.$$

\subsection{$G(3,3,4)$} Finally we consider $G(3,3,4)$. This is a $3$-system with one conjugacy class of reflections, containing $18$ roots lines. Fixing the line spanned by $(1,-1,0,0)$ we find that there are $3$ lines orthogonal to it and $14$ at angle $\pi/3$ to it. It follows that there are $3$ flats of type $A_1 \times A_1$ and $7$ flats of type $A_2$ containing it, and then by counting pairs that there are $27$ flats of type $A_1 \times A_1$ and $42$ flats of type $A_2$ in the arrangement of $G(3,3,4)$. The flats of type $A_1 \times A_1$ form a single $G(3,3,4)$-orbit, while those of type $A_2$ are divided into two orbits. One of these orbits contains the flat spanned by $(1,-1,0,0)$ and $(0,1,-1,0)$ while the other contains the flat spanned by $(1,-1,0,0)$ and $(1,-\zeta,0,0)$. There are $36$ flats in the first orbit and $6$ flats in the second.

We now describe the GS decomposition for each type of flat. First, for $a=(1,-1,0,0)$, $b=(1,0,-1,0)$, and $c=(0,1,-1,0)$ we have $\Delta=\emptyset$, $\Lambda$ of cardinality $6$ and $\Gamma_a$ of cardinality $3$. Thus this $A_2$ is contained in $3$ flats of type $A_3$ and one of type $G(3,3,3)$. Second, for $a=(1,-1,0,0)$, $b=(1,-\zeta,0,0)$, and $c=(1,-\zeta^2,0,0)$ we have $\Delta$ of cardinality $3$, $\Lambda$ of cardinality $12$ and $\Gamma_a=\emptyset$. Thus this $A_2$ is contained in $3$ flats of type $A_2 \times A_1$ and $2$ flats of type $G(3,3,3)$. 

It's straightforward to see that there are no flats of type $A_1 \times A_1 \times A_1$ in this arrangement. To count flats of type $G(3,3,3)$ we count pairs $X \subseteq Y$ where $X$ is of type $A_2$ and $Y$ is of type $G(3,3,3)$ in two ways. Thus if there are $x$ flats of type $G(3,3,3)$ we have $12 x=6 \cdot 2+36$ so that $x=4$. In other words, there are only the four obvious flats of type $G(3,3,3)$, and these form a single conjugacy class. 

To count flats of type $A_3$ we use the same argument. If there are $x$ flats of this type then we have $4x=3 \cdot 36$ so that there are $x=27$ flats of type $A_3$ in the arrangement for $G(3,3,4)$. These flats of type $A_3$ are a single conjugacy class.

Finally, to count flats of type $A_2 \times A_1$ we use the usual argument. Thus if $x$ is the number of these flats we have $x=6\cdot 3=18$. There is a single conjugacy class of flats of this type.

Now we can compute the $\mu$ invariant of $G(3,3,4)$: it is $$\mu(L)=-1+18-27-2 \cdot 42+4 \cdot 16+27 \cdot 6+18 \cdot 2=168.$$ Likewise the $e$-invariant is $$e(L)=240.$$

\section{The groups $W_n(\Gamma,\Delta)$}

\subsection{A product decomposition of $W_n(\Gamma)$} \label{c product formula}
The following lemma is the key point in our first computation of $c_W(t)$ for the groups $W_n(\Gamma)$, and explains the factorization we obtain. It seems likely that for the exceptional groups $W$, the existence of similar product decompositions explains the limited factorizations occurring for $c_W(t)$. 

\begin{lemma}
Each element $w \in W_n(\Gamma)$ may be written uniquely as a product $w=x v$, where $v \in W_{n-1}(\Gamma)$ and 
$$x \in \{ 1 \} \cup \{ \gamma_n (mn) \gamma_n^{-1} \ | \ 1 \leq m < n, \ \gamma \in \Gamma \} \cup \{ \gamma_n \ | \ \gamma \in \Gamma \setminus \{1 \} \}.$$ Moreover the fixed space of $w$ is the intersection of the fixed spaces of $v$ and $x$, and its codimension is the sum of their codimensions.
\end{lemma}
\begin{proof}
For existence, if $e_1,\dots,e_n$ are the canonical basis elements of $\HH^n$, then $w(e_n)=\gamma e_m$ for some $\gamma \in \Gamma$ and $1 \leq m \leq n$. If $m=n$ then we have $\gamma_n^{-1} w \in W_{n-1}(\Gamma)$. Otherwise $\gamma_n^{-1} (mn) \gamma_n w \in W_{n-1}(\Gamma)$. Uniqueness follows by comparing cardinalities.

Evidently the intersection of the fixed spaces of $x$ and $v$ is contained in the fixed space of $w$. Conversely, if $e_1 a_1+e_2 a_2+\cdots+e_n a_n$ belongs to the fixed space of $w=xv$ and $x=\gamma_n$ for some $\gamma \in \Gamma$ then comparing coefficients of $e_n$ on both sides of $$e_1 a_1+\cdots+e_n a_n=xv(e_1a_1+\cdots+e_n a_n)=\cdots+e_n \zeta a_n$$ shows that $x$ fixes $e_1 a_1+\cdots e_n a_n$ and our claim follows. If on the other hand $x=\gamma_n (mn) \gamma_n^{-1}$ then comparing coefficients of $e_m$ on both sides of the same equation shows that $x$ fixes $e_1 a_1+\cdots e_n a_n$. 
\end{proof} 

\begin{corollary}
For $W=W_n(\Gamma)$ with $|\Gamma|=m$, the polynomial $c_W(t)$ is
$$c_W(t)=\prod_{p=1}^n (1+(pm-1) t).$$
\end{corollary}
\begin{proof}
This follows from the previous lemma by induction on $n$.
\end{proof}

\subsection{The polynomial $c_W(t)$ for the groups $W_n(\Gamma,\Delta)$.} Here we observe that for the groups $W=W_n(\Gamma,\Delta)$, the polynomial $c_W(t)$ depends only on $n$ and the orders of $\Gamma$ and $\Delta$. Since $W$ is a complex group (for which the claimed formula is well-known) when $\Gamma$ is cyclic, we obtain the formula in general. First, each element $w \in W_n(\Gamma,\Delta)$ may be written uniquely as $w=\gamma^{(1)}_1 \cdots \gamma^{(n)}_n v$ for $v \in S_n$ and $\gamma^{(1)},\dots,\gamma^{(n)} \in \Gamma$ such that $\gamma^{(1)} \cdots \gamma^{(n)} \in \Delta$. The element $v$ may be decomposed uniquely as a product of disjoint cycles with lengths determining a partition $\lambda$ called the \emph{cycle type} of $v$. Fixing a cycle $c=(p_1 p_2 \cdots p_\ell)$ with $p_1 < p_2 < \cdots p_\ell$ appearing in this decomposition, we call the product $\gamma^{(p_1)} \cdots \gamma^{(p_\ell)}$ of the corresponding elements of $\gamma$ the \emph{cycle product} of $c$. Ordering the cycles of $v$ by length and then breaking ties according to the smallest element in each produces an ordered list $(\pi_1,\pi_2, \dots, \pi_{\ell(\lambda)})$ of the cycle products, and it is straightforward to check the following lemma:
\begin{lemma} 
The fixed space of $w$ is of dimension equal to the number of cycle products in the list $(\pi_1,\pi_2,\dots,\pi_{\ell(\lambda)})$ that are equal to the identity $1$ of $W$.
\end{lemma}

\begin{lemma}
Let $d$ be an integer with $0 \leq d \leq n$. The number of elements of $W_n(\Gamma,\Delta)$ with fixed space of dimension $d$ depends only on $d$ and the orders $|\Gamma|$ and $|\Delta|$ of the groups $\Gamma$ and $\Delta$.
\end{lemma}
\begin{proof}
We count the number $a(\Gamma,\Delta,n,d)$ of elements of $W_n(\Gamma,\Delta)$ with fixed space of dimension $d$ as follows: for each cycle type $\lambda \vdash n$, let $c_\lambda$ be the number of $v \in S_n$ with cycle type $\lambda$. We must count the number of sequences $(\gamma_1,\dots,\gamma_n)$ of elements of $\Gamma$ that have exactly $d$ cycle products (with respect to $v$) equal to the identity. We may count these sequences by first specifying what the cycle products are, and then counting then number of ways to factor each cycle product appropriately. For the first step, we must choose which $d$ of the $\ell(\lambda)$ cycle products are equal to the identity, and then count the number of ways to factor each element of $\Delta$ into $\ell(\lambda)-d$ non-identity elements of $\Gamma$. Finally there are $|\Gamma|^{n-\ell(\lambda)}$ ways to factor the cycle products into products of elements of $\Gamma$ of the appropriate length. We thus have
$$a(\Gamma,\Delta,n,d)=\sum_{\lambda \vdash n} c_\lambda {\ell(\lambda)\choose d} |\Gamma|^{n-\ell(\lambda)} \sum_{\delta \in \Delta} f(\delta,\ell(\lambda)-d), $$ where $f(\delta,k)$ is the number of sequences $(\gamma_1,\dots,\gamma_k)$ of elements of $\Gamma$ such that $\gamma_i \neq 1$ for all $i$ and $\gamma_1 \gamma_2 \cdots \gamma_k=\delta$. This lemma is now a consequence of the next one.
\end{proof}

\begin{lemma}
Fix an integer $m$ There exist functions $a,b: \ZZ_{>0} \to \ZZ_{\geq 0}$ with the following property: for any group $G$ with $|G|=m$, define a function $f(g,k): G \times \ZZ_{>0} \to \ZZ_{ \geq 0}$ by setting $f(g,k)$ equal to the number of sequences $(g_1,g_2,\dots,g_k)$ of elements of $G$ such that $g_1 g_2 \cdots g_k=g$ and $g_i \neq 1$ for all $1 \leq i \leq k$.  Then $f(g,k)=a(k)$ if $g \neq 1$ and $f(g,k)=b(k)$ if $g=1$.
\end{lemma}
\begin{proof}
We must prove that $f(g,k)$ depends only on $k$ and whether or not $g$ is the identity, and not on the specific group $G$. We have $f(1,1)=0$ and $f(g,1)=1$ for all $g \in G$ with $g \neq 1$, establishing the claim for $k=1$. Next observe that the inductive hypothesis implies $$f(1,k)=\sum_{h \in G \setminus \{1 \}} f(h^{-1},k-1)=(|G|-1) a(k-1)$$ and for $g \neq 1$
$$f(g,k)=f(1,k-1)+\sum_{h \in G \setminus \{1,g \}} f(gh^{-1},k-1)=b(k-1)+(|G|-2) a(k-1),$$ establishing the lemma by induction.
\end{proof}

\begin{proposition}
For the group $W=W_n(\Gamma,\Delta)$ with $|\Gamma|=m$ and $|\Delta|=p$, the codimension generating function is 
$$c_W(t)=(1+(m-1)t)(1+(2m-1)t)\cdots(1+((n-1)m-1)t) (1+(np-1)t).$$
\end{proposition} Although this generating function also factors, it is not true that every element of $W_n(\Gamma,\Delta)$ can be written as a product of at most $n$ reflections, so the proof technique used above for $W_n(\Gamma)$ has no chance of working here.

\subsection{A fibration implies the product formula for $p_W(t)$} Here we fix a finite subgroup $\Gamma$ of $\HH^\times$ and let $n$ vary. Write $X_n$ for the hyperplane complement corresponding to the group $W_n(\Gamma)$, which is also the hyperplane complement of $W_n(\Gamma,\Delta)$ when $\Delta \neq 1$ (which is always the case if $\Gamma$ is not abelian). Thus 

$$X_n=\{(x_1,\dots,x_n) \in \HH^n \ | \ x_i \neq 0 \ \text{and} \ x_p \neq \gamma x_q \ \hbox{for all $1 \leq p \neq q \leq n$ and all $\gamma \in \Gamma$} \}.$$

The projection onto the last coordinate defines a map $\pi_n:X_n \to X_{n-1}$. The fiber of $\pi_n$ over $(x_1,\dots,x_{n-1}) \in X_{n-1}$ consists of all points $(x_1,\dots,x_{n-1},x)$ of $X_n$ such that $x \neq 0$ and $x \neq \gamma x_p$ for all $\gamma \in \Gamma$ and all $1 \leq p \leq n-1$. Thus the fiber is the complement in $\HH$ of a set of $1+(n-1)|\Gamma|$ points. In fact it is a fiber bundle, as follows, \emph{mutatis mutandis}, from the proof of Theorem 5.111 in \cite{OrTe}. 

Now we consider the Serre spectral sequence of the fibration $\pi_n$. Since the base is simply connected, the local system of cohomology of the fibers is the constant sheaf $H^q(C)$, where $C$ is the complement in $\HH$ of a set of $1+(n-1)|\Gamma|$ points. In particular $H^q(C)=0$ unless $q=0$ or $q=3$. Moreover, the cohomology of the base $X_{n-1}$ is zero except in degrees divisible by $3$, which implies that the differential $d_2:H^p(X_{n-1},H^q(C)) \to H^{p+2}(X_{n-1},H^{q-1}(C))$ of the $E_2$ page of the spectral sequence is zero and thus that the spectral sequence degenerates. This implies that the Poincar\'e polynomial $p(t)$ of $X_n$ is that of $X_{n-1}$ times $(1+(1+(n-1)|\Gamma|) t)$. It therefore follows by induction on $n$ that the Poincar\'e polynomial of $X_n$, or in other words of $W=W_n(\Gamma,\Delta)$, is 

$$p_W(t)=(1+t)(1+(1+|\Gamma|) t)(1+(1+2 |\Gamma|)t) \cdots (1+(1+(n-1) |\Gamma|) t).$$

\section{The line system of type $Q$} 

\subsection{The line system of $Q$} In the remaining sections we deal with the seven exceptional quaternionic reflection groups of rank at least $3$. We follow Cohen's \cite{Coh} notation, from table II of page 318. Thus to describe the line system of type $Q$, let $$\alpha=\frac{1}{2}(1-i-j-\sqrt{5}k).$$ We consider the group $G(4,2,3)$, which acts on $\HH^3$ by matrix multiplication. The line system for this reflection group consists of the the three coordinate axes (spanned by the vectors $(1,0,0),(0,1,0),$ and $(0,0,1)$) together with the lines spanned by the $12$ vectors
$$(1,1,0),(1,i,0),(1,-1,0),(1,-i,0),(1,0,1),(1,0,i),(1,0,-1),(1,0,-i),(0,1,1),(0,1,i),(0,1,-1),(0,1,-i).$$ Put $v=(1,1,\alpha)$ and consider the $G(4,2,3)$-orbit of the line spanned by $v$. The stabilizer of this line in $G(4,2,3)$ is the group of order four generated by the scalar matrix $-1$ and the transposition interchanging the first two coordinates, and so we obtain $48$ lines in its orbit. We define the line system $Q$ to be this set of $48$ lines together with the $15$ root lines for $G(4,2,3)$. Direct computation shows that the only angles between pairs of lines in $Q$ are $\pi/2$, $\pi/3$, and $\pi/4$. Moreover computing with the vector $(1,0,0)$ shows that there are
\begin{itemize}
\item[(a)] $32$ lines at angle $\pi/3$ to $(1,0,0)$ (all the lines spanned by vectors of the form $(p,q\alpha,r)$ and $(p,q,r\alpha)$ for $p,q,r \in C_4$ with $pqr=-1$),
\item[(b)] $24$ lines at angle $\pi/4$ to $(1,0,0)$ (the remaining $16$ lines in the $G(4,2,3)$-orbit of $v$ together with $8$ of the root line of $G(4,2,3)$), and
\item[(c)] $6$ lines at angle $\pi/2$ to $(1,0,0)$. 
\end{itemize} 

The line spanned by $v$ and the line spanned by $(0,1,-1)$ are at angle $\pi/3$ with one another, and together with the line spanned by $(1,\alpha,1)$ span a $3$-star. Moreover, it follows from (a) above that $(1,0,0)$ is in the same $W=W(Q)$-orbit as the line spanned by $(1,1,\alpha)$. Thus there is only one $W$-orbit of lines in $Q$.

\subsection{The $2$-flats of $Q$} 

The preceding paragraph shows that the $2$-flats containing $(1,0,0)$ are of two sorts:
\begin{itemize}
\item[(a)] Those spanned by $(1,0,0)$ and a line at angle $\pi/4$ to it. The lines in such a $2$-flats form a line system of type $G(4,2,2)$ and there are therefore six of them total, four at angle $\pi/4$ with $(1,0,0)$ and one at angle $\pi/2$. There are $6$ of these containing $(1,0,0)$. 

\item[(b)] Those spanned by $(1,0,0)$ and a line at angle $\pi/3$ to it. These are $3$-stars, and there are $16$ of them containing $(1,0,0)$. 
\end{itemize}

\begin{lemma}
The $2$-flats of $Q$ consist of the following types and no others:
\begin{itemize}
\item[(a)] There are $63$ flats of type $G(4,2,2)$ forming a single $W(Q)$-orbit.

\item[(b)] There are $336$ flats of type $A_2$ forming a single $W(Q)$-orbit.
\end{itemize} The sum of the $\mu$-invariants of these flats is $987$ and the sum of their $e$-invariants is $1239$. Moreover, the flats of type $G(4,2,2)$ are precisely the orthogonal complements of the lines of $Q$.
\end{lemma}
\begin{proof}
Since every line in $Q$ is $W$-conjugate to the span of $(1,0,0)$, the previous computation holds for all lines in $Q$. Letting $x$ be the total number of $2$-flats of type (a) and $y$ the number of type (b), we have
$$6x=63 \cdot 6 \quad \implies \quad x=63 \quad \text{and} \quad 3y=63 \cdot16 \quad \implies \quad y=336.$$ Thus there are $63$ flats of type $G(4,2,2)$, each with $\mu$-invariant $5$, and $336$ flats of type $A_2$, each with $\mu$-invariant $2$, contained in the arrangement of $Q$. The orthogonal complement to $(1,0,0)$ is evidently a flat of type $G(4,2,2)$, so the orthogonal complement to any line of $Q$ is a $G(4,2,2)$ and this proves the last statement as well as the claim that these form a single conjugacy class. That the flats of type $A_2$ form a single conjugacy class can be worked out explicitly, but it follows from \cite{BST} as well. Finally, the statements about $\mu$ and $e$-invariants are immediate from what we have proved.
\end{proof}

Since $|W(Q)|=12096$ the $\mu$ and $e$-invariants of $Q$ are
$$\mu(Q)=1-63+5 \cdot 63+2 \cdot 336=925 \quad \text{and} \quad e(Q)=12096-9 \cdot 63-2 \cdot 336-63-1=10793,$$ implying that $p_W(t)$ and $c_W(t)$ are as claimed in the introduction.

\section{The quaternionic reflection group of type $R$} 

\subsection{The line system} We follow the beautiful paper \cite{Tit} of J. Tits. Our aim here to compute $c_W(t)$ and $p_W(t)$. To achieve this we first compute, for a fixed line $\ell$ in the line system $R$, the set of all $2$-dimensional subspaces there are spanned by $\ell$ and another $\ell'$ in $R$, and for each of them compute the number of lines it contains. According to Tits \cite{Tit}, Corollary 2 on page 65, the line system of type $R$ contains a single orbit of lines consisting of: 
\begin{itemize}
\item[(a)] Three lines spanned by the vectors $(2,0,0)$, $(0,2,0)$, and $(0,0,2)$,
\item[(b)] $24$ lines spanned by vectors $(x,y,z)$ with $\{|x|, |y|,|z| \}=\{0,\sqrt{2},\sqrt{2} \}$,
\item[(c)] $192$ lines spanned by vectors $(x,y,z)$ with $\{|x|, |y|,|z| \}=\{1,1,\sqrt{2} \}$,
\item[(d)] $96$ lines spanned by vectors $(x,y,z)$ with $\{|x|, |y|,|z| \}=\{1,\tau,\tau^{-1} \}$, where $\tau=\frac{1+\sqrt{5}}{2}$.
\end{itemize}

\subsection{The $2$-flats of $R$} Fixing the line $\ell_0$ spanned by $(2,0,0)$, computing using the preceding shows that the other lines of $R$ are organized into the following types according to their angle with $\ell_0$:
\begin{itemize}
\item[(a)] $10$ lines orthogonal to $\ell_0$,
\item[(b)] $160$ lines at angle $\pi/3$ to $\ell_0$,
\item[(c)] $80$ lines at angle $\pi/4$ to $\ell_0$,
\item[(d)] $32$ lines at angle $\pi/5$ to $\ell_0$, and
\item[(e)] $32$ lines at angle $2 \pi /5$ to $\ell_0$.
\end{itemize} Moreover, the explicit description of $R$ in \cite{Coh} shows that the line system orthogonal to a line of $R$ is the line system of $W_2(Q,\{\pm 1\})$. In particular there is a $W(R)$-orbit of $2$-flats consisting of $315$ flats of type $W_2(Q,\{ \pm 1\})$. Hence if $x$ is the number of these containing $\ell_0$ we have $315 x=315 \cdot 10$ so $x=10$. These flats account for all $10$ lines orthogonal to $\ell_0$, proving that these are the only flats of type $W_2(Q,\{\pm 1 \})$ arising for $R$. Moreover they account for all $10 \cdot 8=80$ of the lines at angle $\pi/4$ to $\ell_0$. Now by the results of \cite{BST}, the only other maximal parabolic subgroups of $W(R)$ are of type $A_2$ and $G(5,5,2)$. It follows that the $160$ lines at angle $\pi/3$ to $\ell_0$ are organized into $80$ flats of type $A_2$ containing it, implying that there are $105 \cdot 80$ flats of type $A_2$, and the remaining lines at angle $\pi/5$ and $2 \pi /5$ are accounted for by $16$ flats of type $G(5,5,2)$ containing $\ell_0$. It then follows that if $x$ is the number of flats of type $G(5,5,2)$ we have $5x=315 \cdot 16$, so that $x=63 \cdot 16$. Thus in $R$ there are:
\begin{itemize}
\item[(a)] $2^4 \cdot 3 \cdot 5^2 \cdot 7$ flats of type $A_2$ with $\mu$ and $e$-invariants $2$,
\item[(b)] $2^4 \cdot 3^2 \cdot 7$ flats of type $G(5,5,2)$ with $\mu$ and $e$-invariants $4$, and
\item[(c)] $3^2 \cdot 5 \cdot 7$ flats of type $W_2(Q,\{\pm 1\})$, with $\mu$ invariant $9$ and $e$-invariant $21$.  
\end{itemize} There are no other flats, \cite{BST} implies that each of these types is a single conjugacy class, and the sum of their $\mu$-invariants is 
$$2 \cdot 2^4 \cdot 3 \cdot 5^2 \cdot 7+ 4 \cdot 2^4 \cdot 3^2 \cdot 7+9 \cdot 3^2 \cdot 5 \cdot 7=23667$$ while the sum of their $e$-invariants is $27447$.  Finally, the $\mu$-invariant of $R$ is
$$\mu(R)=23667-315+1=23353$$ and since $|W(R)|=1209600$ its $e$-invariant is
$$e(R)=1209600-2 \cdot 2^4 \cdot 3 \cdot 5^2 \cdot 7-4 \cdot 2^4 \cdot 3^2 \cdot 7-21 \cdot 3^2 \cdot 5 \cdot 7-315-1=1181837.$$ This shows that $p_W(t)$ and $c_W(t)$ are as claimed in the introduction.

\section{The quaternionic reflection group of type $S_1$}

\subsection{The line system of type $S_1$}

Let $V=\HH^4$. $S_1$ is the set of all lines spanned by vectors of the form $\pm e_i \pm e_j$ for $1 \leq i < j \leq 4$ together with the $G(2,2,4)$-orbit of the vector $\frac{1}{\sqrt{2}}(1,i,j,k)$. This orbit consists of all vectors of the form $\frac{1}{\sqrt{2}}(\pm \epsilon_1,\pm \epsilon_2,\pm \epsilon_3,\pm \epsilon_4)$, where $(\epsilon_1,\epsilon_2,\epsilon_3,\epsilon_4)$ is a permutation of $(1,i,j,k)$ and there are an even number of $-$ signs. All these vectors are of square norm $2$, and it is straightforward to check that for each pair of vectors in this set, the angle between them is either $\pi/2$ or $\pi/3$, so $S_1$ is a $3$-system. 

\subsection{The nine frames in $S_1$} The lines in $S_1$ may be partitioned into nine \emph{frames}, by which we mean a set of $4$ mutually orthogonal lines. It follows incidentally that the property of being orthogonal is an equivalence relation on $S_1$, and that the transitive permutation action of $W(S_1)$ on the set of lines is imprimitive with these nine blocks. The frames may be described as follows: they are

$$\{(1,-1,0,0) \HH ,(1,1,0,0) \HH ,(0,0,1,-1) \HH ,(0,0,1,1) \HH \} $$ and the three permutations of this, as well as
$$\{(1,i,j,k) \HH, (1,-i,-j,k) \HH, (1,-i,j,-k) \HH, (1,i,-j-k) \HH \}$$ and the six permutations of this set. The group $W(S_1)$ acts transitively on the set of frames (this produces a nine-dimensional permutation representation, on which the reflections act via conjugates of $(12)(34)(56)(78)$).

\subsection{The $3$-stars containing $(1,i,j,k)$} Here we simply list the $3$-stars of $S_1$ that contain the line spanned by $(1,i,j,k)$. Since this line belongs to each of them, we list these by giving two vectors that span the other two lines in the $3$-star. They are
\begin{align*}
&\{(1,-1,0,0),(1,-i,k,-j)\},\{(1,1,0,0),(1,-i,-k,j)\},\{(1,0,-1,0),(1,-k,-j,i)\},\{(1,0,1,0),(1,k,-j,-i)\}, \\ 
&\{(1,0,0,-1),(1,j,-i,-k)\},\{(1,0,0,1),(1,-j,i,-k)\},\{(0,1,-1,0),(1,j,i,k)\},\{(0,1,1,0),(1,-j,-i,k)\}, \\
&\{(0,1,0,-1),(1,k,j,i) \}, \{(0,1,0,1),(1,-k,j,-i)\},\{(0,0,1,-1),(1,i,k,j)\},\{(0,0,1,1),(1,i,-k,-j)\}, \\
&\{(1,j,k,i),(1,k,i,j)\}, \{(1,-j,-k,i),(1,k,-i,-j)\},\{(1,-j,k,-i),(1,-k,-i,j)\},\{(1,j,-k,-i),(1,-k,i,-j)\}.
\end{align*} To verify, for instance, that $\{(1,i,j,k),(1,j,k,i),(1,k,i,j)\}$ is a $3$-star, apply the reflection in the line $(1,j,i,k)$ to the star $(0,1,-1,0),(0,0,1,-1),(0,1,0,-1)$.

\subsection{The two-dimensional spaces spanned by pairs of lines in $S_1$} Each pair of distinct lines in $S_1$ produces a two-dimensional $\HH$-subspace of $V=\HH^4$. If the two lines are orthogonal to one another, they are the only two lines in the subspace they generate, while if they are at angle $\pi/3$ to one another then there is a third line in the subspace at angle $\pi/3$ to both of them. Since there are $9$ frames, each with $6$ possible choices of two orthogonal vectors, there are $54$ subspaces of the first type. To count the subspaces of the second type we observe that we may count pairs $(\ell,X)$, where $\ell$ is a line of $S_1$ and $X$ is a subspace containing $\ell$ and two other lines at angle $\pi/3$ to it in two ways: one the one hand, each line is contained in precisely one frame consisting of $4$ lines, and it follows that the other $32$ lines are partitioned into $16$ pairs of lines spanning, together with $\ell$, a space $X$ of the desired type. Thus there are $16$ possible $X$'s for each $\ell$, and the number of pairs is therefore $36 \cdot 16$. On the other hand, it is $3$ times the number of such subspaces $X$, so there must be $12 \cdot 16=192$ such subspaces $X$. This proves all the statements in the following lemma except the one about the conjugacy classes of flats of type $A_2$, which one verifies directly using the description of the $3$-stars given above.
\begin{lemma}
There are $54$ two-dimensional subspaces of $V=\HH^4$ spanned by two orthogonal lines of $S_1$ and not containing any other lines of $S_1$, thus producing flats of type $A_1 \times A_1$, and $192$ two-dimensional subspaces of $V$ spanned by three lines of $S_1$ at angle $\pi/3$ to one another (and not containing any other line of $S_1$), thus producing flats of type $A_2$. There is only one $W(S_1)$-orbit of flats of type $A_1 \times A_1$, but there are four $W(S_1)$-orbits of flats of type $A_2$, each of cardinality $48$, with representatives given by the following four flats: the flat spanned by $(1,-1,0,0)$ and $(0,1,-1,0)$, the flat spanned by $(0,0,1,-1)$ and $(1,i,j,k)$, the flat spanned by $(0,0,1,-1)$ and $(1,j,i,k)$, and the flat spanned by $(0,0,1,-1)$ and $(1,k,i,j)$. 
\end{lemma}

The sum of the $\mu$-invariants (or $e$-invariants, since all $2$-flats are real) for the $2$-flats of $S_1$ is therefore
$$\sum_{\mathrm{rk}(X)=2} \mu(X)=54+2 \cdot 192=438.$$

\subsection{The three-dimensional flats of $S_1$} Suppose now that $X$ is a $3$-flat of $S_1$ (a three dimensional subspace of $V=\HH^4$ that is spanned by three lines of $S_1$). There are three mutually exclusive possibilities for $X$: 

\begin{itemize}
\item[Type I] $X$ contains only lines orthogonal to one another, is of type $A_1 \times A_1 \times A_1$. There are $36$ flats of this type, all conjugate to one another since they are the orthogonal complements of the $36$ lines of $S_1$.
\item[Type II] $X$ contains at least one pair of orthogonal lines and at least one pair of lines at angle $\pi/3$ and is of type $A_3$; 
\item[Type III] $X$ contains only lines at angle $\pi/3$ to one another, is of type $G(3,3,3)$. 
\end{itemize} For type I, we may take the span of any three lines belonging to the same frame. There are therefore $36=4 \cdot 9$ subspaces of type I, each of which contains three $A_1 \times A_1$'s. Before describing the type II and III $3$-flats in detail we give an example of each:
\begin{itemize}
\item[(a)] The six vectors $\{(1,-1,0,0),(1,0,-1,0),(1,0,0,-1),(0,1,-1,0),(0,1,0,-1),(0,0,1,-1)\}$ span six distinct lines belonging to $S_1$, and the space generated by these six lines is a $3$-flat of type II containing precisely these six lines of $S_1$. 
\item[(b)] The nine vectors $$\{(0,1,-1,0),(0,1,0,-1),(0,0,1,-1),(1,i,j,k),(1,j,i,k),(1,i,k,j),(1,k,j,i),(1,j,k,i),(1,k,i,j)\}$$ span nine distinct lines belonging to $S_1$, and the space generated by these $9$ lines is a $3$-flat of type III containing precisely these nine lines of $S_1$.
\end{itemize}
\begin{lemma}
\item[(a)] Each $3$-flat of $S_1$ of type $A_3$ contains precisely six lines, partitioned into three disjoint $A_1 \times A_1$'s, and contains precisely four $A_2$'s. Each $A_1 \times A_1$ of $S_1$ is contained in precisely eight $3$-flats of type II, and there are exactly $144$ total $3$-flats of type II. These are divided into four $W(S_1)$-orbits according to the four $W(S_1)$-orbits on the flats of type $A_2$ (all flats of type $A_2$ in a given flat of type $A_3$ are conjugate to one another). 
\item[(b)] Each $3$-flat of $S_1$ of type $G(3,3,3)$ contains precisely nine lines and precisely twelve $A_2$'s. Each $3$-star of $S_1$ is contained in precisely four such $3$-flats, and there are exactly $64$ $3$-flats of type $G(3,3,3)$. These form a single $W(S_1)$-orbit.
\end{lemma}
\begin{proof}
This count follows the same pattern as usual. First consider the flat of type $A_1 \times A_1$ spanned by the vectors $(0,1,-1,0)$ and $(0,1,1,0)$. Of the $34$ lines of $S_1$ not contained in this flat precisely two are perpendicular to both vectors $(0,1,-1,0)$ and $(0,1,1,0)$ and the remaining $32$ lines have angle $\pi/3$ with both. It follows that these $32$ lines are partitioned into $8$ subsets corresponding to the flats of type $A_3$ containing our given $A_1 \times A_1$. Now counting pairs $(X,Y)$ where $X$ is of type $A_1 \times A_1$ and $Y$ is of type $A_3$ in two ways shows that $54 \cdot 8=3x$ where $x$ is the total number of flats of type $A_3$, implying $x=18 \cdot 8=144$ as claimed. 

The count of the flats of type $G(3,3,3)$ we compute the GS decomposition corresponding to a flat of type $A_2$. We illustrate this with the flat spanned by $(0,1,-1,0)$ and $(0,0,1,-1)$. Since $(x,y,z,w)$ is orthogonal to this flat if and only if $y=z=w$, there are no lines of $S_1$ orthogonal to this flat. The lines that are orthogonal to $(0,1,-1,0)$ are precisely the other three lines in its frame, spanned by $(0,1,1,0)$, $(1,0,0,-1)$, and $(1,0,0,1)$. Similarly there are precisely $3$ lines orthogonal to each of the other lines in our $A_2$. The remaining $24$ lines must therefore be partitioned into $4$ sets of $6$ lines each, corresponding to $5$ flats of type $G(3,3,3)$ containing the given $A_2$. Exactly the same reasoning holds for the other orbits of flats of type $A_2$, so if $x$ is the number of flats of type $G(3,3,3)$ we have
$$12x=4\cdot 192 \quad \implies \quad x=64$$ as claimed. The results of \cite{BST} imply the claims about orbits.
\end{proof}

It follows that the sum of the $\mu$-invariants of the $3$-flats is
$$\sum_{\mathrm{rk}(X)=3} \mu(X)=36+6\cdot144+16 \cdot 64=1924$$ while the sum of the $e$-invariants is
$$\sum_{\mathrm{rk}(X)=3} e(X)=36+6 \cdot 144+20 \cdot 64=2180.$$ Hence the $\mu$- and $e$-invariants of $S_1$ are
$$\mu(S_1)=-1+36-438+1924=1521 \quad \text{and} \quad e(S_1)=6912-2180-438-36-1=4257$$ and hence $p_W(t)$ and $c_W(t)$ are as claimed in the introduction.

\section{The quaternionic reflection group of type $S_2$} 

\subsection{The line system} The line system for this group is the union of the line system for the Weyl group of type $F_4$ and the orbit, under $W(F_4)$, of the line spanned by $(1,i,j,k)$. The resulting group $W(S_2)$ has one orbit of $72$ root lines and hence one conjugacy class of reflections. Fixing the line spanned by $(1,0,0,0)$, the remaining lines may be divided into classes according to the angle they make with this line. There are:
\begin{enumerate}
\item[(a)] Nine lines perpendicular to $(1,0,0,0)$, spanned by the vectors $\epsilon_i$ and $\epsilon_i \pm \epsilon_j$ for $2 \leq i < j \leq 4$.

\item[(b)] Six lines at angle $\pi/4$ to $(1,0,0,0)$, spanned by the vectors $\epsilon_1 \pm \epsilon_i$ for $2 \leq i \leq 4$.

\item[(c)] Fifty-six lines at angle $\pi/3$ to $(1,0,0,0)$, spanned by the vectors $(1,\pm 1,\pm 1, \pm 1)$ and $(1, \pm p, \pm q, \pm r)$ with $\{p,q,r \}=\{i,j,k\}$. 
\end{enumerate}

\subsection{The $2$-flats} We fix a root line, which up to the action of $W(S_2)$ we may assume to be spanned by $\epsilon_1=(1,0,0,0)$. By direct calculation we find that the span of $\epsilon_1$ and another line in the line system $S_2$ is of one of the following four types:
\begin{enumerate}
\item[(a)] The direct product $A_1 \times A_1$ for the lines spanned by vectors of the form $\epsilon_i \pm \epsilon_j$ with $2 \leq i < j \leq 4$. There are six of these.
\item[(b)] A line system of type $B_2$ for the lines spanned by vectors of the form $\epsilon_1 + \epsilon_j$ with $2 \leq j \leq 4$. There are three of these.
\item[(c)] A line system of type $A_2$ for the lines spanned by vectors of the form $(1,\pm 1,\pm 1, \pm 1)$  There are four such $2$-flats.
\item[(d)] A line system of type $A_2$ for the lines spanned by vectors of the form $(1, \pm p,\pm q,\pm r)$ with $\{p,q,r\}=\{i,j,k\}$. There are twenty-four such $2$-flats.
\end{enumerate} Moreover, except for the flats of type $A_2$, each type of $2$-flat is a single orbit for the action of $W(S_2)$. The flats of type $A_2$ fall into two orbits. The flat containing the lines spanned by $(1,0,0,0),(1,i,j,k)$, and $(1,-i,-j,-k)$ has the property that its orthogonal complement does not contain a root line for $S_2$. The flat containing the lines spanned by $(1,0,0,0),(1,1,1,1)$, and $(1,-1,-1,-1)$ has orthogonal complement consisting of vectors of the form $(0,x,y,z)$ with $x+y+z=0$, and hence is another flat of type $A_2$. Since the stabilizer of $(1,0,0,0)$ in $W(S_2)$ contains at least all signed permutations of the last three coordinates, every $2$-flat of type $A_2$ that contains the line spanned by $(1,0,0,0)$ is conjugate to one of these two. Since every line is conjugate to $(1,0,0,0)$, it follows that there are precisely these two orbits of $2$-flats of type $A_2$.

Now by counting pairs $(X,Y)$ where $X$ is a root line and $Y$ is a $2$-flat of a given type containing $X$ we obtain the following:

\begin{enumerate}
\item[(a)] There are $216$ $2$-flats of type $A_1 \times A_1$.
\item[(b)] There are $54$ $2$-flats of type $B_2$.
\item[(c)] There are $96$ $2$-flats of type $A_2$ whose orthogonal complement is another flat of type $A_2$.
\item[(d)] There are $576$ $2$-flats of type $A_2$ whose orthogonal complement does not contain a root line.
\end{enumerate}

The sum of the $\mu$-invariants (equivalently, $e$-invariants) for these two flats is therefore
$$\sum_{\mathrm{rk}(X)=2} \mu(X)=216+3\cdot54+2\cdot96+2\cdot 576=1722$$

\subsection{The $3$-flats} We will see below that there are the following types of $3$-flats in $S_2$:
\begin{enumerate}
\item[(a)] $288$ $3$-flats of type $A_1 \times A_2$,
\item[(b)] $864$ $3$-flats of type $A_3$,
\item[(c)] $72$ $3$-flats of type $B_3$, 
\item[(d)] $256$ $3$-flats of type $G(3,3,3)$, and
\item[(e)] $108$ $3$-flats of type $G(4,4,3)$.
\end{enumerate}

Given this, the sum of their $\mu$-invariants is
$$\sum_{\mathrm{rk}(X)=3} \mu(X)=2 \cdot 288+6\cdot864+15\cdot72+16\cdot256+30\cdot108=14176.$$ And finally, the $\mu$-invariant of $\HH^4$ for this line system is
$$-1+72-1722+14176=12525.$$ Thus the Poincar\'e polynomial for $S_2$ is
$$p(t)=1+72t+1722t^2+14176t^3+12525t^4=(1+t)(1+25t)(1+46t+501t^2).$$

We begin by analyzing which $3$-flats arise by adjoining a line to a $2$-flat of type $B_2$. Since all such $2$-flats are conjugate, we may assume that our $B_2$ is the one spanned by $\epsilon_1=(1,0,0,0)$ and $\epsilon_2=(0,1,0,0)$. 
\begin{enumerate}
\item[(a)] Adjoining the line spanned by a vector of the form $(1,\pm p,\pm q,\pm r)$ gives a $3$-flat of type $G(4,4,3)$ (this can be checked by observing that the equation for the resulting subspace is of the form $x_4=\pm r q^{-1} x_3$). There are six $3$-flats of this form containing the given $B_2$.
\item[(b)] Adjoining $\epsilon_3,\epsilon_4$, $\epsilon_3-\epsilon_4$, or $\epsilon_3+\epsilon_4$ gives a $3$-flat of type $B_3$. There are four of these.
\end{enumerate}

Now counting pairs $(X,Y)$ consisting of a $2$-flat $X$ of type $B_2$ and a  $3$-flat $Y$ of type $G(4,4,3)$ such that $X \subseteq Y$ implies that there are $108$ $3$-flats of type $G(4,4,3)$ in $S_2$. Similar reasoning shows that there are $72$ $3$-flats of type $B_3$ in $S_2$. 

Next we consider which $3$-flats arise by adjoining a line to a $2$-flat of type $A_1 \times A_1$. Again, all such $2$-flats are conjugate, so we may assume that our $A_1 \times A_1$ is the one spanned by $\epsilon_1$ and $\epsilon_2-\epsilon_3$. The remaining $70$ lines of $S_2$ are then partitioned according to the $3$-flat that they, together with $A_1 \times A_1$, generate. The result is:
\begin{enumerate}
\item[(a)] Adjoining a line spanned by a vector of the form $(1,\pm p ,\pm q, \pm r)$ for $\{p,q,r \}=\{i,j,k\}$ produces a $3$-flat of type $A_3$. There are $12$ of these containing the given $A_1 \times A_1$.
\item[(b)] Adjoinining a line with last coordinate $x_4=0$ or with $x_2+x_3=0$ gives a $3$-flat of type $B_3$. There are two of these containing the given $A_1 \times A_1$.
\item[(c)] The remaining lines produce four $A_2 \times A_1$'s containing the given $A_1 \times A_1$.
\end{enumerate} 

Now we consider the $2$-flats of type $A_2$ with the property that the orthogonal complement is another $2$-flat of type $A_2$. Since these are all conjugate we may as well assume that we are working with the $2$-flat spanned by $(0,1,-1,0)$ and $(0,0,1,-1)$. 
\begin{enumerate}
\item[(a)] Adjoining a line spanned by a vector of the form $(1,\pm p , \pm q, \pm r)$ gives a $2$-flat containing precisely the lines of the original $A_2$ plus the six permutations of $(1,\pm p,\pm q,\pm r)$. This is a $3$-flat of type $G(3,3,3)$, and there are therefore $8$ three-flats of type $G(3,3,3)$ containing the given $2$-flat.
\item[(b)] There are three root lines perpendicular to the given $A_2$, and adjoining any one of them produces an $A_2 \times A_1$. There are therefore three subspaces of this type containing the given $2$-flat.
\item[(c)] The remaining $18$ lines are accounted for by three $3$-flats of type $B_3$ containing the given $2$-flat of type $A_2$. 
\end{enumerate}

Finally we consider the $2$-flats of type $A_2$ whose orthogonal complement does not contain any line. We will consider the result of adjoining a line to the span of $(1,0,0,0)$ and $(1,i,j,k)$. 

\begin{enumerate}
\item[(a)] Adjoining one of $(0,1,0,0)$, $(0,0,1,0)$, or $(0,0,0,1)$ produces a $3$-flat with twelve lines, probably of type $G(4,4,3)$. There are three of these containing the given $2$-flat.
\item[(b)] Adjoining a lines of the form $(1,\pm 1,\pm 1,\pm 1)$ produces a $3$-flat with nine lines, probably of type $G(3,3,3)$. There are four of these containing the given $2$-flat.
\item[(c)] The remaining $18$ lines are accounted for by six $3$-flats of type $A_3$ containing the given $2$-flat.
\end{enumerate}

\section{The quaternionic reflection group of type $S_3$}

\subsection{The line system of type $S_3$} The set of root lines consists of those for the group $W_4(Q,\{\pm 1\})$ together with the $W_4(Q,\{\pm 1\})$-orbit of the line spanned by $(1,1,1,1)$. The lines for $W_4(Q,\{\pm 1\})$ are spanned by the vectors $(1,0,0,0)$, $(0,1,0,0)$, $(0,0,1,0)$, and $(0,0,0,1)$ together with all lines spanned by vectors of the form $e_i+e_j q$ where $1 \leq i < j \leq 4$ and $q \in Q$. The angle between any two of these lines is one of the three numbers $\pi/2$, $\pi/3$, and $\pi/4$. According to Cohen \cite{Coh} these lines form a single $W(S_3)$-orbit.

\subsection{The $2$-flats of $S_3$} We fix the line $\ell_0$ spanned by $(1,0,0,0)$ and consider the span of this line together with one more line of $S_3$. We note that the orthogonal complement of $\ell_0$ is a flat of type $W_3(Q,\{ \pm 1\})$, which acts in the obvious fashion on the last three coordinates. 

Firstly, taking the span of $\ell_0$ and a line spanned by a vector of the form $(0,1,p,0)$ or $(0,1,0,p)$ or $(0,0,1,p)$ gives a flat of type $A_1 \times A_1$ containing only these two lines. There are $24$ such flats containing $\ell_0$, forming a single $W_3(Q,\{\pm 1\}$-orbit. It follows that there are $24 \cdot 90$ flats of type $A_1 \times A_1$ in $S_3$, forming a single $W(S_3)$-orbit.

Secondly, taking the span of $\ell_0$ and any one of the lines spanned by the vectors $(0,1,0,0)$, $(0,0,1,0)$, or $(0,0,0,1)$ produces a $2$-flat of type $W_2(Q,\{\pm 1\}$. There are $3$ such flats containing $\ell_0$, all conjugate by its stabilizer, and it follows that there is a single $W(S_3)$-orbit of flats of type $W_2(Q,\{\pm 1 \})$ containing $54$ flats.

Finally, taking the span of $\ell_0$ and a line spanned by a vector of the form $(1,p,q,r)$ with $p,q,r \in Q$ and $pqr=\pm 1$ gives a $2$-flat of type $A_2$. There are $128$ such lines, producing $64$ flats of type $A_2$ containing $\ell_0$, all conjugate by the stabilizer $W_3(Q,\{\pm 1\})$ of $(1,0,0,0)$. Hence there is a single $W(S_3)$-orbit of flats of type $A_2$, containing precisely $60 \cdot 64$ flats. 

\begin{lemma}
For the line system of type $S_3$, there are the following $2$-flats and no others, each type comprising a singe $W(S_3)$-orbit:
\begin{itemize}
\item[(a)] $2 \cdot 3^3$ flats of type $W_2(Q,\{\pm1 \})$,
\item[(b)] $2^8 \cdot 3 \cdot 5$ flats of type $A_2$, and
\item[(c)]  $2^4 \cdot 3^3 \cdot 5$ flats of type $A_1 \times A_1$.
\end{itemize}
The sum of the $\mu$-invariants of these flats is $10326$ and the sum of their $e$-invariants is $10974$. Moreover, the orthogonal complement of a flat of a given type is another flat of the same type.
\end{lemma}
\begin{proof}
All but the last assertion follows from what we have done above. The last assertion follows from the previous ones by noting that it is true for the particular flats spanned by: $(1,0,0,0)$ and $(0,1,0,0)$; $(0,1,-1,0)$ and $(0,0,1,-1)$; $(1,0,0,0)$ and $(0,1,i,0)$.
\end{proof}

\subsection{The $3$-flats of $S_3$} 

\begin{lemma}
For the line system of type $S_3$, there are the following $3$-flats and no others, each type comprising a single $W(S_3)$-orbit:
\begin{itemize}
\item[(a)] $2^2 \cdot 3^2 \cdot 5$ flats of type $W_3(Q,\{\pm 1\})$, 
\item[(b)] $2^9 \cdot 5$ flats of type $G(3,3,3)$,
\item[(c)] $2^7 \cdot 3^3 \cdot 5$ flats of type $A_3$, and
\item[(d)] $2^8 \cdot 3^2 \cdot 5$ flats of type $A_2 \times A_1$.
\end{itemize} The sum of the $\mu$-invariants of these flats is $195220$ and the sum of their $e$-invariants is $272420$. Moreover, the flats of type $W_3(Q,\{\pm 1\})$ are precisely the orthogonal complements of the lines of $S_3$.
\end{lemma}

\begin{proof}
To prove (a), we note that given the flat spanned by the vectors $(1,0,0,0)$ and $(0,1,0,0)$, the orthogonal complement consists of all vectors of the form $(0,0,a,b)$. Now given any line $\ell$ of $S_3$, there is always a line of $S_3$ spanned by a vector of the form $(0,0,a,b)$ that is orthogonal to $\ell$, as can be checked case-by-case. It follows that the span of $\ell$ and the vectors $(1,0,0,0)$ and $(0,1,0,0)$ is contained in the orthogonal complement of some line of $S_3$, which is a flat of type $W_3(Q,\{ \pm 1 \})$. Thus the $170$ lines of $S_3$ that are not contained in the flat spanned by $(1,0,0,0)$ and $(0,1,0,0)$ are partitioned into $10$ sets of $17$ lines each, corresponding to $10$ flats of type $W_3(Q,\{\pm 1\}$ containing our given flat of type $W_2(Q,\{\pm 1\})$. Since all $W_2(Q,\{\pm 1\})$'s of $S_3$ are conjugate by $W(S_3)$, if there are $x$ total flats of type $W_3(Q,\{\pm 1\})$ then we have $3x=2 \cdot 3^3 \cdot 10$ so that $x=2^2 \cdot 3^2 \cdot 5$ as claimed. It also follows from this argument that each flat of type $W_3(Q,\{\pm 1\})$ is the orthogonal complement of a line of $S_3$, and in particular there is one $W(S_3)$-orbit of such flats.

We now fix the flat of type $A_2$ spanned by $(0,1,-1,0)$ and $(0,0,1,-1)$. It follows from our count in part (a) and the results of the previous lemma that this is contained in $3$ flats of type $W_3(Q,\{\pm 1\})$. Adjoining any line spanned by a vector of the form $(1,p,q,r)$ for $p,q,r$ equal to $i,j,k$ up to a sign produces a flat of type $G(3,3,3)$; there are $8$ of these. Adjoining a vector of the form $(1,q,0,0)$ for $q \in Q$ with $q \neq \pm 1$ gives a flat of type $A_3$, as does adjoining a line spanned by a vector of the form $(1,q,q,1)$ or $(1,q,q,-1)$ for $q \in Q$ with $q \neq \pm 1$; there are thus $18$ flats of type $A_3$ containing our given $A_2$. Finally, each of the three lines spanend by $(1,1,1,1)$, $(1,-1,-1,-1)$ and $(1,0,0,0)$ produces an $A_2 \times A_1$, and we have now accounted for all $177$ lines of $S_3$ not contained in our given $A_2$. The counts in cases (b), (c) and (d) now follow by the same idea as above. 

To see that there are no other flats, one computes using the preceding calculation the number of each type of flat counted so far containing a given $A_1 \times A_1$, and observes that this accounts for all $178$ lines of $S_3$ not contained in the given $A_1 \times A_1$. Finally, the claim about conjugacy follows from the calculation in \cite{BST}, and it is routine to check that the sum of the $\mu$-invariants is as claimed. 
\end{proof}

\subsection{The $\mu$- and $e$-invariants of $S_3$} Combining the results proved above shows that the $\mu$-invariant of $S_3$ is
$$\mu(S_3)=195220-10326+180-1=185073$$ and since $|W(S_3)|=3317760$ Its $e$-invariant is
$$e(S_3)=3317760-272420-10974-180-1=3034185.$$ This establishes that $p_W(t)$ and $c_W(t)$ are as claimed in the introduction.

\section{The quaternionic reflection group of type $T$}

\subsection{The line system of type $T$}

The line system of type $T$ is the union of the line system of type $H_4$ and the $W(H_4)$-orbit of the line spanned by the vector $(1,i,j,k)$. Somewhat more explicitly, it consists of the lines spanned by the vectors
$$(1,0,0,0),(0,1,0,0),(0,0,1,0),(0,0,0,1),(1,\pm 1, \pm 1, \pm 1),$$ together with all vectors obtained from $(\tau^{-1},1,\tau,0)$ where $\tau=\frac{1+\sqrt{5}}{2}$  by even permutations of the coordinates and arbitrary sign changes (this is the line system $H_4$), together with the lines spanned by the vectors 
$$(1,i^p,j^p,k^p) \quad \text{and} \quad (1,(-i)^p,(-j)^p,(-k)^p)$$ for $p \in I$ in the binary icosahedral group $I$ (this is the orbit of the line spanned by $(1,i,j,k)$ under $W(H_4)$). There are a total of $180$ lines in $T$, forming a single $W(T)$-orbit.

\subsection{The $2$-flats} 

\begin{lemma}
The $2$-flats of $T$ consist of:
\begin{itemize}
\item[(a)] $2 \cdot 3^3 \cdot 5^2$ flats of type $A_1 \times A_1$, forming a single $W(T)$-orbit,
\item[(b)] a $W(T)$-orbit consisting of $2^3 \cdot 3 \cdot 5^2$ flats of type $A_2$, characterized by the property that the orthogonal complement of one is another flat of type $A_2$ (in this same $W(T)$-orbit),
\item[(c)] a $W(T)$-orbit consisting of $2^4 \cdot 3^2 \cdot 5^2$ flats of type $A_2$, characterized by the property that the orthogonal complememnt of one contains no line of $T$,
\item[(d)] $2^3 \cdot 3^3$ flats of type $G(5,5,2)$, comprising a single $W(T)$-orbit. The orthogonal complement of one of these is another flat of type $G(5,5,2)$. 
\end{itemize} The sum of the $\mu$ invariants of these is equal to the sum of their $e$-invariants,
$$2 \cdot 3^3 \cdot 5^2 + 2^3 \cdot 3 \cdot 5^2 \cdot 2+2^4 \cdot 3^2 \cdot 5^2 \cdot 2+2^3 \cdot 3^3 \cdot 4=10614.$$
\end{lemma}
\begin{proof}
The line spanned by $(1,0,0,0)$ is at angle $\pi/3$ to the lines spanned by the vectors $(1,(\pm i)^p,(\pm j)^p, (\pm k)^p)$ for $p \in I$, the vectors $(1,\pm 1, \pm 1, \pm 1)$, and the vectors 
$$(1,\pm \tau, \pm \tau^{-1},0) \quad (1, \pm \tau^{-1},0,\pm \tau), \quad (1,0, \pm \tau,\pm \tau^{-1}).$$

It is orthogonal to the lines spanned by the vectors $(0,1,0,0)$, $(0,1,0,0)$, $(0,0,0,1)$, and the vectors 
$$(0,\tau^{-1}, \pm \tau, \pm 1), \quad (0,\tau, \pm 1, \pm \tau^{-1}), \quad \text{and} \quad (0,1,\pm \tau^{-1}, \pm \tau).$$ Finally, it is at angle $\pi/5$ to the lines spanned by the vectors
$$(\tau, \pm \tau^{-1}, \pm 1, 0), \quad (\tau, \pm 1, 0, \pm \tau^{-1}), \quad \text{and} \quad (\tau,0,\pm \tau^{-1}, \pm 1)$$ and at angle $2 \pi /5$ to the lines spanned by the vectors
$$(\tau^{-1}, \pm 1, \pm \tau, 0), \quad (\tau^{-1}, \pm \tau, 0 , \pm 1), \quad \text{and} \quad (\tau^{-1},0, \pm 1, \pm \tau).$$ It follows that the only $2$-flats in $T$ are of type $A_2$, $A_1 \times A_1$, and $G(5,5,2)$. 

We begin by analyzing the $A_2$'s. First we observe that the span of $(1,0,0,0)$ and a vector of the form $(1,(\pm i)^p,(\pm j)^p,(\pm k)^p)$ has the property that its orthogonal complement does not contain a line of $T$. This produces $60$ flats of type $A_2$ containing $(1,0,0,0)$ with this property. The remaining $20$ lines of $T$ at angle $\pi/3$ to $(1,0,0,0)$ account for $10$ flats of type $A_2$ with the property that the orthogonal complement is another flat of type $A_2$. The counts of these two types of flats in (b) and (c) follow the usual argument. The fact that the $A_2$ flats of one of these two types are all conjugate to one another follows by using the $W(H_3)$-action on the set of lines orthogonal to $(1,0,0,0)$. 

Now we analyze the $A_1 \times A_1$'s. There is one of these for each line orthogonal to $(1,0,0,0)$, and since the set of them spans a flat of type $H_3$ they are all conjugate to one another. So each line of $T$ is contained in $15$ flats of type $A_1 \times A_1$, all such flats are conjugate to one another, and the count in part (a) follows the usual pattern.

Finally we deal with the flats of type $G(5,5,2)$ containing $(1,0,0,0)$. There are $6$ of these, accounting for the $24$ lines of $T$ at angle $\pi/5$ or $2 \pi /5$ with $(1,0,0,0)$. They are all conjugate by the action of $W(H_3)$, and the claims in (d) follow.
\end{proof}

\subsection{The $3$-flats}

\begin{lemma}
The $3$-flats of $T$ consist of
\begin{itemize}
\item[(a)] $2^2 \cdot 3^2 \cdot 5$ flats of type $H_3$, which are the orthogonal complements of the lines of $T$ and form a single $W(T)$-orbit,
\item[(b)] $2^3 \cdot 3^2 \cdot 5^2$ flats of type $A_2 \times A_1$, comprising a single $W(T)$-orbit,
\item[(c)] $2^3 \cdot 3^3 \cdot 5$ flats of type $G(5,5,2) \times A_1$, comprising a single $W(T)$-orbit,
\item[(d)] $2^5 \cdot 3^3$ flats of type $G(5,5,3)$, comprising a single $W(T)$-orbit,
\item[(e)] $2^5 \cdot 5^3$ flats of type $G(3,3,3)$, comprising a single $W(T)$-orbit,
\item[(f)] a $W(T)$-orbit consisting of $2^2 \cdot 3^2 \cdot 5^2$ flats of type $A_3$, characterized by the property that an $A_2$ contained in one of them has orthogonal complement another flat of type $A_2$,
\item[(g)] a $W(T)$-orbit consisting of $2^2 \cdot 3^3 \cdot 5^3$ flats of type $A_3$, characterized by the property that an $A_2$ contained in one of them does not contain any line of $T$ in its orthogonal complement.
\end{itemize} The sum of the $\mu$-invariants of these flats is
$$2^2 \cdot 3^2 \cdot 5 \cdot 45+2^3 \cdot 3^2 \cdot 5^2 \cdot 2+2^3 \cdot 3^3 \cdot 5 \cdot 4+2^5 \cdot 3^3 \cdot 48 +2^5 \cdot 5^3 \cdot 16 + 2^2 \cdot 3^2 \cdot 5^2 \cdot 6+2^2 \cdot 3^3 \cdot 5^3 \cdot 6=207892$$ and the sum of their $e$-invariants is
$$2^2 \cdot 3^2 \cdot 5 \cdot 45+2^3 \cdot 3^2 \cdot 5^2 \cdot 2+2^3 \cdot 3^3 \cdot 5 \cdot 4+2^5 \cdot 3^3 \cdot 72+2^5 \cdot 5^3 \cdot 20+2^2 \cdot 3^2 \cdot 5^2 \cdot 6+2^2 \cdot 3^3 \cdot 5^3 \cdot 6=244628.$$
\end{lemma}

\begin{proof}
We begin by determining the number of flats of type $A_2 \times A_1$. Each of these contains exactly one flat of type $A_2$ whose orthogonal complement contains at least one line of $T$, so it follows from the previous lemma that there are $3 \cdot 2^3 \cdot 3 \cdot 5^2=2^3 \cdot 3^2 \cdot 5^2$ of them, as claimed in (b). Since the orthogonal complement of an $A_2$ is another line system of type $A_2$, the $A_2 \times A_1$'s containing a given $A_2$ are a single orbit under the group generated by reflections in lines orthogonal to it, and it follows that these are a single $W(T)$-orbit. This gives (b).

The assertions about flats of type $G(5,5,2) \times A_1$ in (c) follow the same pattern: the flats of type $G(5,5,2) \times A_1$ containing a given $G(5,5,2)$ are in bijection with the lines of $T$ contained in its orthogonal complement; there are $5$ of these spanning another flat of type $G(5,5,2)$ and hence there are $2^3 \cdot 3^3 \cdot 5$ flats of type $G(5,5,2) \times A_1$ forming a single $W(T)$-orbit. 

Now we analyze the $3$-flats containing the $A_1 \times A_1$ of $T$ spanned by $(1,0,0,0)$ and $(0,1,0,0)$. First, each $H_3$ containing it has a unique line orthogonal to the given $A_1 \times A_1$. There are only two such lines in $T$: those spanned by $(0,0,1,0)$ and $(0,0,0,1)$, and each of them does indeed produce an $H_3$ containing our $A_1 \times A_1$. Thus there are precisely two $H_3$'s containing a given $A_1 \times A_1$. From the preceding counts there are four flats of each of the type $A_2 \times A_1$ and $G(5,5,2) \times A_1$ containing our $A_1 \times A_1$. The $A_2 \times A_1$'s are spanned by our $A_1 \times A_1$ together with the vectors of the form $(0,1,\pm \tau^{-1},\pm \tau)$ (accounting for two of them) and of the form $(1,0,\pm \tau,\pm \tau^{-1})$ (accounting for the other two). The $G(5,5,2) \times A_1$'s are spanned by our $A_1 \times A_1$ together with the vectors of the form $(\tau,0,\pm \tau^{-1}, \pm 1)$ and $(\tau^{-1},0,\pm 1,\pm \tau)$ (accounting for two of them) and of the form $(0,\tau,\pm 1,\pm \tau^{-1})$ and $(0,\tau^{-1},\pm \tau, \pm 1)$ (accounting for the other two).

The remaining lines of $T$ that are not in any of the $3$-flats containing $(1,0,0,0)$ and $(0,1,0,0)$ accounted for so far are those spanned by the vectors of the form $(1,\pm 1,\pm 1,\pm 1)$ and $(1,(\pm i)^p,(\pm j)^p,(\pm k)^p)$ for $p \in I$. The first sort account for two $A_3$'s with the property that an $A_2$ contained in one of them has orthogonal complement another $A_2$, and the second sort account for $30$ flats of type $A_3$ with the property that an $A_2$ contained in one of them is not orthogonal to any line of $T$. 

Now the counts of the flats in (a), (f), and (g) follows, as well as the fact that the flats of type $H_3$ are a single $W(T)$-orbit, and that each of them is the orthogonal complement of a line of $T$. Moreover, it also follows that the flats of type $A_3$ with the property that the $A_2$'s they contain have orthogonal complement another flat of type $A_2$ are a single $W(T)$-orbit.

Now we consider the flat of type $A_2$ spanned by $(1,0,0,0)$ and $(1,1,1,1)$. The orthogonal complement of this is another flat of type $A_2$. The counts we have already established imply that this flat is contained in $3$ flats of type $H_3$, $3$ flats of type $A_2 \times A_1$, and $6$ flats of type $A_3$. Together these account for all the lines of $T$ except for the $120$ lines of the form $(1,(\pm i)^p,(\pm j)^p, (\pm k)^p)$ for $p \in I$, which produce $20$ flats of type $G(3,3,3)$ containing our $A_2$. 

Let $X$ be the flat of type $G(5,5,2)$ spanned by $(1,0,0,0)$ and $(\tau^{-1},1,\tau,0)$. The orthogonal complement of this flat is another flat of type $G(5,5,2)$, spanned by the vectors $(0,0,0,1)$ and $(0,\tau,-1,\tau^{-1})$. It follows that this flat is contained in five flats of type $G(5,5,2) \times A_1$, all conjugate to one another, and five flats of type $H_3$. These account for all the lines of $T$ except for those spanned by vectors of the form $(1,(\pm i)^p,(\pm j)^p,(\pm k)^p)$. There are $120$ such lines, partitioned into $12$ subsets each producing a $G(5,5,3)$ containing our $G(5,5,2)$. The count in part (d) follows from this.

Finally we consider the flat of type $A_2$ spanned by $(1,0,0,0)$ and $(1,i,j,k)$, which does not contain any line of $T$ in its orthogonal complement. This flat is therefore not contained in any flat of type $A_2 \times A_1$ or $H_3$. It follows from our count of the flats of type $A_3$ in (g) that it contained in $15$ such $A_3$'s, and from our count of the flats of type $G(5,5,3)$ that it is contained in $6$ of them. Together these account for $117$ lines of $T$, leaving $60$ lines partitioned into $10$ subsets each producing a flat of type $G(3,3,3)$. From this and our calculation of the number of $G(3,3,3)$'s containing the other sort of $A_2$, it follows that if $x$ is the number of flats of $T$ of type $G(3,3,3)$ then we have
$$12x=10 \cdot 2^4 \cdot 3^2 \cdot 5^2+20 \cdot 2^3 \cdot 3 \cdot 5^2 \quad \implies \quad x=4000.$$ This establishes the count in (e). Finally, all the remaining statements about conjugacy follow from the results of \cite{BST} (though we could, with a bit more effort, establish them in a computer-free fashion), and the calculation of the sum of the $\mu$-invariants is a consequence of the counts we have achieved.

\end{proof}

\subsection{The $\mu$-invariant and Poincar\'e polynomial} 

The preceding results together with $|W(T)|=2592000$ imply that the $\mu$- and $e$-invariants of $T$ are
$$\mu(T)=207892-10614+180-1=197457 \quad \text{and} \quad e(T)=2592000-244628-10614-180-1=2336577.$$ This shows $p_W(t)$ and $c_W(t)$ are as claimed in the introduction.

\section{The quaternionic reflection group of type $U$}

\subsection{} The largest dimension in which an exceptional quaternionic (non-complex) reflection group exists is $5$. Here we will describe this group. First, we will abbreviate by writing

$$\omega=\frac{1}{2}(-1+i+j+k).$$

According to Cohen, the group $W(U)$ is generated by reflections in the hyperplanes orthogonal to the $165$ lines spanned by the following vectors:

$$(1,0,0,0,0), \ (0,1,0,0,0), (0,0,1,0,0), \ (0,0,0,1,0), \ (0,0,0,0,1),$$

$$(0,1,\omega,\omega,1), \ (0,1, i \omega i^{-1},k \omega k^{-1},i), \ (0,1, j \omega j^{-1}, i \omega i^{-1},j), \ (0,1,k \omega k^{-1},j \omega j^{-1},k)$$ and all cyclic permutations and arbitrary sign changes of these. 

We remind the reader that we are writing vectors as row vectors only to save space. With this in mind, the unitary matrix

$$\left( \begin{matrix} \frac{1}{\sqrt{2}} & -\omega^2/\sqrt{2} & 0 & 0 & 0 \\ 
- \frac{1}{\sqrt{2}} & -\omega^2/\sqrt{2} & 0 & 0 & 0 \\  
 0& 0  & -\omega^2/\sqrt{2} &  \frac{1}{\sqrt{2}} & 0 \\
 0& 0  & \omega^2/\sqrt{2} &  \frac{1}{\sqrt{2}} & 0 \\ 
 0 & 0 & 0 & 0 & -\omega^2 \end{matrix} \right) $$ transforms the set $U$ of lines into the set of lines spanned by the following collection of vectors: the line system $D_4$, embedded in the first four coordinates, the line spanned by the vector $\left( \begin{matrix} 0 & 0 & 0 & 0 & 1 \end{matrix} \right)$, and the images by $W(D_4)$ of the lines spanned by the following vectors:
 
 $$\left( \begin{matrix} 1, & i, & j, & k, & 0 \end{matrix} \right), \quad \left( \begin{matrix} 1, & 1, & 1, & \omega-\omega^2, & \sqrt{2} \end{matrix} \right),$$
 
 $$\left( \begin{matrix} 1, & 1, & 1, & i-j-k, & i \sqrt{2} \end{matrix} \right), \quad \left( \begin{matrix} 1, & 1, & 1, & -i+j-k, & j \sqrt{2} \end{matrix} \right), \quad \text{and} \quad \left( \begin{matrix} 1, & 1, & 1, & -i-j+k, & k \sqrt{2} \end{matrix} \right). $$ In most of this section we will use this realization of $U$. It has the advantage of clearly displaying $S_1$ as the orthogonal complement of a line of $U$, and at the same time clearly displaying the complex line system $K_5$ as a subsystem of $U$ (see for instance chapter 7 of Lehrer and Taylor \cite{LeTa} for the relevant realization of $K_5$).

 \subsection{The flats of $U$} We note that the orthogonal complement to the vector $(0,0,0,0,1)$ is the line system of type $S_1$. Thus the $W(U)$ orbit of this system consists of $165$ four-dimensional flats all of type $S_1$. We now begin our calculation of the Poincar\'e polynomial with a description of the set of $2$-flats of the arrangement of $U$.
 
\begin{lemma}
There is a single $W(U)$-orbit of flats of $U$ of type $A_2$, consisting of $5 \cdot 11 \cdot 64$ flats, and a single $W(U)$-orbit of flats of type $A_1 \times A_1$, consisting of $2 \cdot 27 \cdot 5 \cdot 11$ flats. The sum of the $\mu$-invariants of these is equal to the sum of their $e$-invariants, $$2 \cdot 5 \cdot 11 \cdot 64+2 \cdot 27 \cdot 5 \cdot 11=2 \cdot 5 \cdot 11 \cdot 91=10010.$$
\end{lemma}
\begin{proof}
First we count the number of each type of $2$-flat. We note that the angle between two lines of $U$ is either $\pi /2$ or $\pi /3$. Supposing now that $\ell_0$ is the line spanned by the vector $(0,0,0,0,1)$, there are $36$ lines orthogonal to it (the lines of $S_1$) and the remaining $128$ lines of $U$ are therefore at angle $\pi /3$ to it. It follows that there are $36$ flats of type $A_1 \times A_1$ containing $\ell_0$ and $64$ flats of type $A_2$ containing $\ell_0$. Letting $x$ be the number of flats of type $A_1 \times A_1$, it follows that the number of pairs $(\ell,X)$ consiting of a line $\ell$ of $U$ and a flat $X$ of type $A_1 \times A_1$ containing $\ell$ is $165 \cdot 36 = 2x$, so that $x=18 \cdot 165$ as claimed. Similarly, if we now let $x$ be the number of flats of type $A_2$, we have $165 \cdot 64=3x$, implying $x=64 \cdot 55$ as claimed.

The group $W(S_1)$ acts transitively on its set of lines. It follows from this that all the $A_1 \times A_1$'s containing $\ell_0$ are conjugate to one another, and since $W(U)$ acts transitively on its lines we conclude that all $A_1 \times A_1$'s are conjugate to one another. The group $W(S_1)$ also acts transitively on the lines at angle $\pi/3$ to $\ell_0$. To prove this, we compute the effect of the reflection $r$ with root line $(1,i,j,k,0)$ on the vector $(1,1,1,\omega-\omega^2,\sqrt{2})$. The result is
$$r(1,1,1,\omega-\omega^2,\sqrt{2})=(j,1-i+k,-j,j,\sqrt{2}),$$ which is proportional to $(-1,-i+j-k,1,-1,j\sqrt{2})$. It follows that the lines spanned by the vectors $(1,1,1,\omega-\omega^2,\sqrt{2})$ and $(1,1,1,-i+j-k,j\sqrt{2})$ are in the same $W(S_1)$-orbit. By cyclic symmetry in $i,j,k$ the lines spanned by the vectors $(1,1,1,i-j-k,i\sqrt{2})$ and $(1,1,1,-i-j+k,k\sqrt{2})$ also belong to this orbit, and our claim follows from the fact that $W(D_4) \leq W(S_1)$.
\end{proof}

Before moving on the the $3$-flats we record the GS decomposition of $U$ with respect to the $2$-flat of type $A_2$ spanned by the root lines $a=(1,-1,0,0,0)$, $b=(1,0,-1,0,0)$, and $c=(0,1,-1,0,0)$. The set $\Delta$ of lines of $U$ orthogonal to all three consists of lines spanned by vectors with first three coordinates equal. In other words $\Delta$ consists of the nine lines spanned by the vectors:
$$(1,1,1,\omega-\omega^2,\sqrt{2}), \quad (-1,-1,-1,-\omega+\omega^2,\sqrt{2}), \quad (1,1,1,i-j-k,i \sqrt{2}), \quad (-1,-1,-1,-i+j+k,i \sqrt{2}), $$

$$(1,1,1,-i+j-k,j \sqrt{2}), \quad (-1,-1,-1,i-j+k,j \sqrt{2}), \quad (1,1,1,-i-j+k,k \sqrt{2}), \quad (-1,-1,-1,i+j-k,k \sqrt{2}), $$ and $(0,0,0,0,1)$. This is a $3$-system in a $3$-dimensional space with the property that it contains a line $(0,0,0,0,1)$ not orthogonal to any other of its lines. It follows that this is a $G(3,3,3)$. 

Now we specify the set $\Lambda$ of $72$ lines of $U$ not orthogonal to any of $a,b,$ and $c$. This consists of the $W(D_4)$-orbit (acting on the first four coordinates) of the vector $(1,i,j,k,0)$ together with the lines spanned by the vectors in the $S_3$-orbits (acting by permuting the first three coordinates) of 
$$(1,-1,\omega-\omega^2,-1,\sqrt{2}), \quad (1,-1,-\omega+\omega^2,1,\sqrt{2}), \quad (1,-1,i-j-k,-1,i \sqrt{2}), \quad (1,-1,-i+j+k, 1, i \sqrt{2}), $$
$$(1,-1,-i+j-k,-1,j \sqrt{2}), \quad (1,-1,i-j+k,1,j \sqrt{2}), \quad (1,-1,-i-j+k,-1,k \sqrt{2}), \quad (1,-1,i+j-k,1,k \sqrt{2}) $$

Finally we observe that the set $\Gamma_a$ of lines of $U$ that are perpendicular to $a$ but not to $b$ or $c$ consists of the $3$ lines spanned by the vectors

$$(1,1,0,0,0), \quad (0,0,1,1,0), \quad (0,0,1,-1,0)$$ together with the $24$ lines spanned by
$$(1,1,x,1,y\sqrt{2}), \ (1,1,-1,-x,y \sqrt{2}), \ (1,1,-x,-1,y\sqrt{2}), \ (-1,-1,x,1,y\sqrt{2}), \ (-1,-1,1,x,y\sqrt{2}), \ (-1,-1,-x,-1,y \sqrt{2}),$$ where $x$ ranges over the set $\omega-\omega^2,i-j-k,-i+j-k,-i-j+k$ and $y=1$ if $x=\omega-\omega^2$, $y=i$ if $x=i-j-k$, $y=j$ if $x=-i+j-k$, and $y=k$ if $x=-i-j+k$.

For the $3$-flats we have the following:

\begin{lemma}
There is a unique $W(U)$-orbit of each of the following types of $3$-flats: 
\begin{itemize}
\item[(a)] $16 \cdot 27 \cdot 5 \cdot 11$ flats of type $A_3$ with $\mu$-invariant $6$,
\item[(b)] $64 \cdot 5 \cdot 11$ flats of type $G(3,3,3)$ with $\mu$-invariant $16$ and $e$-invariant $20$, 
\item[(c)] $64 \cdot 9 \cdot 5 \cdot 11$ flats of type $A_2 \times A_1$ with $\mu$-invariant $2$, and
\item[(d)] $2 \cdot 27 \cdot 5 \cdot 11$ flats of type $A_1 \times A_1 \times A_1$ with $\mu$-invariant $1$.
\end{itemize} Moreover, the flats of type $A_1 \times A_1 \times A_1$ are precisely the orthogonal complements of the flats of type $A_1 \times A_1$, and the flats of type $G(3,3,3)$ are precisely the orthogonal complements of the flats of type $A_2$. The sum of the $\mu$-invariants of these flats is
$$32 \cdot 81 \cdot 5 \cdot 11+1024\cdot 5 \cdot 11+128\cdot 9 \cdot 5 \cdot 11+2\cdot 27 \cdot 5 \cdot 11=265210$$ and the sum of their $e$-invariants is
$$32 \cdot 81 \cdot 5 \cdot 11+256\cdot 25 \cdot 11+128\cdot 9 \cdot 5 \cdot 11+2\cdot 27 \cdot 5 \cdot 11=279290.$$
\end{lemma}

\begin{proof}
We first count the $3$-flats of each type, beginning with the $3$-flats of type $A_3$. To count these we count pairs $X \subseteq Y$ where $X$ is of type $A_2$ and $Y$ is of type $A_3$. If $x$ is the number of flats of type $A_3$ we obtain, using the preceding GS decomposition and the fact that each $A_3$ contains four flats of type $A_2$, 

$$4x=27 \cdot 5 \cdot 11 \cdot 64 \quad \implies \quad x=16 \cdot 27 \cdot 5 \cdot 11,$$ as claimed.

Next we count the $3$-flats of type $G(3,3,3)$. Now setting $x$ equal to the number of these the same argument yields 

$$12x=12  \cdot 5 \cdot 11 \cdot 64 \quad \implies \quad x=5 \cdot 11 \cdot 64,$$ as claimed.

Now letting $x$ be the number of $3$-flats of type $A_2 \times A_1$ we get
$$x=9 \cdot 5 \cdot 11 \cdot 64,$$ establishing the formula in (c).

And finally, letting $x$ be the number of $3$-flats of type $A_1 \times A_1 \times A_1$, we observe that the orthogonal complement to the $A_1 \times A_1$ spanned by $(1,0,0,0,0)$ and $(0,1,0,0,0)$ in Cohen's realization of $U$ given above contains precisely the three lines spanned by $(0,0,1,0,0),$ $(0,0,0,1,0)$, and $(0,0,0,0,1)$, so that
$$3x=3 \cdot 2 \cdot 27 \cdot 5 \cdot 11 \quad \implies \quad x=2 \cdot 27 \cdot 5 \cdot 11,$$ establishing the formula in (d). 

Observing that the number of flats of type $G(3,3,3)$ is equal to the number of flats of type $A_2$ and that, by our description of the GS decomposition for such an $A_2$, the orthogonal complement of a flat of type $A_2$ is a flat of type $G(3,3,3)$, it follows that the flats of type $G(3,3,3)$ are all conjugate to one another and are precisely the orthogonal complements of the flats of type $A_2$. Likewise the flats of type $A_1 \times A_1 \times A_1$ are precisely the orthogonal complements of the flats of type $A_1 \times A_1$ and hence are all conjugate to one another. That the $A_2 \times A_1$'s are all conjugate to one another follows from the fact that the $A_2$'s are all conjugate, and the orthogonal complement of an $A_2$ is a $G(3,3,3)$, which acts transitively on its set of lines.  

It remains to establish that the $A_3$'s are a single $W(U)$-orbit. This can be done by again considering again the GS decomposition of $U$ with respect to the $A_2$ spanned by $(1,-1,0,0,0)$ and $(0,1,-1,0,0)$ described above: it suffices to prove that the group $G(3,3,3)$ generated by the reflections in the lines of $\Delta$ (those orthogonal to the given $A_2$) acts transitively on the $27$ lines of $\Gamma_a$. This a straightforward calculation.
\end{proof}

Finally, we describe the $4$-flats of $U$:

\begin{lemma}
The arrangement of the line system $U$ contains the following $4$-flats and no others:
\begin{itemize}
\item[(a)] $3 \cdot 5 \cdot 11$ flats of type $S_1$ with $\mu$-invariant $1521$ and $e$-invariant $4257$,
\item[(b)] $128 \cdot 5 \cdot 11$ flats of type $G(3,3,4)$ with $\mu$-invariant $168$ and $e$-invariant $240$,
\item[(c)] $128 \cdot 27 \cdot 11$ flats of type $A_4$ with $\mu$-invariant $24$,
\item[(d)] $64 \cdot 3 \cdot 5 \cdot 11$ flats of type $G(3,3,3) \times A_1$ with $\mu$-invariant $16$ and $e$-invariant $20$, and
\item[(e)] $16 \cdot 27 \cdot 5 \cdot 11$ flats of type $A_3 \times A_1$ with $\mu$-invariant $6$.
\end{itemize} The sum of the $\mu$-invariants of these flats is
$$1521 \cdot 3 \cdot 5 \cdot 11+168 \cdot 128 \cdot 5 \cdot 11+24\cdot 128 \cdot 27 \cdot 11+16 \cdot 64 \cdot 3 \cdot 5 \cdot 11+6 \cdot 16 \cdot 27 \cdot 5 \cdot 11=2657589$$ and the sum of their $e$-invariants is
$$4257 \cdot 3 \cdot 5 \cdot 11+240 \cdot 128 \cdot 5 \cdot 11+24\cdot 128 \cdot 27 \cdot 11+20 \cdot 64 \cdot 3 \cdot 5 \cdot 11+6 \cdot 16 \cdot 27 \cdot 5 \cdot 11=3658149.$$ Moreover, each type of $4$-flat above is a single $W(U)$-orbit.
\end{lemma}

\begin{proof}
We begin with the count of the flats of type $S_1$. We have seen above that a flat of type $S_1$ contains $36$ flats of type $A_1 \times A_1 \times A_1$. On the other hand, we will check that each flat of $U$ of type $A_1 \times A_1 \times A_1$ is contained in precisely two flats of $U$ of type $S_1$. It follows that if the arrangement of type $U$ contains $x$ flats of type $S_1$, then $36x=2 \cdot 2 \cdot 27 \cdot 5 \cdot 11$ so that $x=165$ as claimed. To see that each $A_1 \times A_1 \times A_1$ is contained in precisely two $S_1$'s we observe that the $S_1$'s containing an $A_1 \times A_1 \times A_1$ are in bijection with the lines of $U$ that are orthogonal to it. In one direction, given an $S_1$ containing our $A_1 \times A_1 \times A_1$, there is a unique fourth line in $S_1$ that is orthogonal to the given triple of orthogonal lines. In the other directino, in Cohen's realization it is enough to check that for the span of $(1,0,0,0,0), (0,1,0,0,0),$ and $(0,0,1,0,0)$, each line of $U$ orthogonal to these three is contained in an $S_1$ that contains them, which is easy. This also proves that every $S_1$ is the orthogonal complement of a line of $U$, and hence that the $S_1$'s in $U$ form a single $W(U)$-orbit.

The other $3$-systems of rank $4$ that contain $A_1 \times A_1 \times A_1$ are $A_2 \times A_1 \times A_1$, $A_1^{\times 4}$, and $A_3 \times A_1$. We note that $A_1^{\times 4}$ does not occur in $U$ since the span of $A_1^{\times 3}$ together with either of the lines orthogonal to it is a copy of $S_1$. Likewise, the orthogonal complement of an $A_2$ of $U$ is a flat of type $G(3,3,3)$, which does not contain two orthogonal lines, so $A_2 \times A_1 \times A_1$ does not occur as a flat of $U$. This leaves only $A_3 \times A_1$. Since we have seen above that a $A_1 \times A_1 \times A_1$ is contained in two $S_1$'s accounting for a total of $66$ lines, the remaining $96$ lines must be accounted for by $24$ flats of type $A_3 \times A_1$ containing a given $A_1 \times A_1 \times A_1$. Thus if $x$ is the number of $A_3 \times A_1$'s in $U$, then $3x=24 \cdot 2 \cdot 27 \cdot 5 \cdot 11$ so that $x=16 \cdot 27 \cdot 5 \cdot 11$. Since this is equal to the number of $A_3$'s, it follows that each $A_3$ is contained in a unique $A_3 \times A_1$, and since there is one $W(U)$-orbit of $A_3$'s there is also a unique $W(U)$-orbit of $A_3 \times A_1$'s. 

Next we consider the $4$-flats in $U$ that contain a $G(3,3,3)$. The only $3$-systems of rank $4$ that contain a flat of type $G(3,3,3)$ are $S_1$, $G(3,3,4)$, and $G(3,3,3) \times A_1$. By using the usual counting argument we find that if $x$ is the number of $S_1$'s of $U$ containing a given $G(3,3,3)$ then we have $64 \cdot 5 \cdot 11 x=165 \cdot 64$ so that $x=3$. There is one $G(3,3,3) \times A_1$ containing a given $G(3,3,3)$ for each line of $L$ in its orthogonal complement; since this complement is a flat of type $A_2$ there are $3$ of these. Thus the total number of $G(3,3,3) \times A_1$'s is $3 \cdot 64 \cdot 5 \cdot 11$. Moreover, of the $156$ lines of $U$ in the complement of a given flat $X$ of type $G(3,3,3)$, the preceding calculation shows that $81$ of them are accounted for by $S_1$'s containing $X$ and $3$ are accounted for by $G(3,3,3) \times A_1$'s. The remaining $72$ lines must therefore be partitioned into $8$ sets of $9$ lines each, corresponding to $8$ flats of type $G(3,3,4)$ containing $X$. It follows that if there are $x$ total flats of type $G(3,3,4)$, then we have $4x=8 \cdot 64 \cdot 5 \cdot 11$, so that $x=128 \cdot 5 \cdot 11$. 

We have now achieved the desired count of the $4$-flats in the arrangement of $U$ in all cases except $A_4$ and $A_2 \times A_2$. We first analyze $A_4$. To do this we will first count the number of $A_4$'s containing a given flat $X$ of type $A_3$. There are a total of $159$ lines not in the given $A_3$. We have already seen that there is a unique flat of type $A_3 \times A_1$ containing each $A_3$. Since we know how many flats of type $G(3,3,4)$ and $S_1$ there are we can deduce the number of each that contains a given $A_3$: each flat of type $G(3,3,4)$ contains $27$ flats of type $A_3$ so $27 \cdot 128 \cdot 5 \cdot 11=16 \cdot 27 \cdot 5 \cdot 11 x$ where $x$ is the number of $G(3,3,4)$'s containing a given $A_3$. This shows that there are $x=8$ such $G(3,3,4)$'s. Likewise each flat of type $S_1$ contains $144$ flats of type $A_3$, so $165 \cdot 144=16 \cdot 27 \cdot 5 \cdot 11 x$ where now $x$ is the number of $S_1$'s containing a given $A_3$. Hence there is a unique $S_1$ containing each $A_3$. The $S_1$ accounts for $30$ additional lines, the $G(3,3,4)$'s account for $8 \cdot 12=96$ lines, and the $A_3 \times A_1$ for one additional line. Therefore the remaining $32$ lines must come from $8$ flats of type $A_4$ containing the given $A_3$. Finally it follows that the number $x$ of flats of type $A_4$ is determined by $8 \cdot 16 \cdot 27 \cdot 5 \cdot 11=5x$, so that $x=128 \cdot 27 \cdot 11$, achieving the count for the flats of type $A_4$. 

To finish the count we must show that there are no flats of type $A_2 \times A_2$. Any such flat contains a flat of type $A_2 \times A_1$. On the other hand, we can count the number of flats of each of the other types containing $A_2 \times A_1$, using the fact that we know how many there are of each in total. The results are: each $A_2 \times A_1$ is contained in $4$ flats of type $G(3,3,4)$, in $12$ flats of type $A_4$, in $3$ flats of type $A_3 \times A_1$, and in $4$ flats of type $G(3,3,3) \times A_1$. This accounts for all $161$ lines of $U$ that are not in the given $A_2 \times A_1$, so it cannot be contained in any additional $4$-flats and there are therefore no flats of type $A_2 \times A_2$.

According to \cite{BST}, page 13, each of the types of flats of dimension $4$ is a single conjugacy class, establishing the last claim.

\end{proof}

By combining these lemmas we obtain the $\mu$-invariant of $U$:

$$\mu(U)=1-165+10010-265210+2657589=2402225=5^2 \cdot 7^2 \cdot 37 \cdot 53.$$

Hence the Poincar\'e polynomial is

$$1+165t+10010t^2+265210 t^3+2657589t^4+2402225t^5=(1+t)(1+25t)(1+37t)(1+49t)(1+53t).$$ By page 320 of Cohen \cite{Coh}, the order of $W(U)$ is $27371520$, so the preceding calculations also imply that its codimension generating function is
$$c_W(t)=1+165t+10010t^2+279290t^3+3658149t^4+23423905t^5.$$

\def\cprime{$'$} \def\cprime{$'$}

\end{document}